\documentclass[12pt,paper=a4,headsepline,headings=normal,abstract=true]{scrartcl}

\usepackage[english]{babel}
\usepackage[utf8]{inputenc}
\usepackage[T1]{fontenc} 
\usepackage{textcomp}
\usepackage[margin=2cm]{geometry}
\usepackage[reqno]{amsmath}  
\usepackage{amssymb}          
\usepackage{amsthm}    
\usepackage{framed}    
\usepackage{amstext}         
\usepackage{enumerate}
\usepackage{enumitem}	  
\usepackage{bbm}
\usepackage{pifont}           
\usepackage{scrpage2}
\usepackage{graphicx}
\usepackage{verbatim}
\usepackage{mathrsfs}
\usepackage{url}
\usepackage{accents}
\urldef{\urluni}{\url}{http://www.mathematik.uni-kl.de/fuana/}
\urldef{\emailgrothaus}{\url}{grothaus@mathematik.uni-kl.de}
\urldef{\emailvosshall}{\url}{vosshall@mathematik.uni-kl.de}

\setkomafont{pageheadfoot}{%
\normalfont\normalcolor\itshape\small
}
\setkomafont{pagenumber}{\normalfont\bfseries}

\makeatletter\@addtoreset{equation}{section}\makeatother

\allowdisplaybreaks

\theoremstyle{plain}      \newtheorem{theorem}{Theorem}[section]
                          \newtheorem{corollary}[theorem]{Corollary}
                          \newtheorem{proposition}[theorem]{Proposition}

\theoremstyle{remark}     \newtheorem{remark}[theorem]{Remark}
                          \newtheorem{lemma}[theorem]{Lemma}
													\newtheorem{example}[theorem]{Example}

\theoremstyle{definition} \newtheorem{definition}[theorem]{Definition}

\begin{document} 

\newcommand{\grad}{\nabla}
\newcommand{\D}{\partial}
\newcommand{\E}{\mathcal{E}}
\newcommand{\N}{\mathbb{N}}
\newcommand{\R}{\mathbb{R}_{\scriptscriptstyle{\ge 0}}}
\newcommand{\dom}{\mathcal{D}}
\newcommand{\ess}{\operatorname{ess~inf}}
\newcommand{\cem}{\operatorname{\text{\ding{61}}}}
\newcommand{\supp}{\operatorname{\text{supp}}}
\newcommand{\ca}{\operatorname{\text{cap}}}

\setenumerate[1]{label=(\roman*)}       
\setenumerate[2]{label=(\alph*)}

\begin{titlepage}
\title{\Large Integration by parts on the law of \\ the modulus of the Brownian bridge}
\author{\normalsize \textsc{Martin Grothaus} \footnote{University of Kaiserslautern, P.O.Box 3049, 67653
Kaiserslautern, Germany.} \thanks{\urluni}~\thanks{\emailgrothaus}
 \and \normalsize \textsc{Robert Voßhall} \footnotemark[1] \footnotemark[2]~\thanks{\emailvosshall}}
\date{\small \today}
\end{titlepage}
\maketitle
\pagestyle{headings}

\begin{abstract}
\noindent
We prove an infinite dimensional integration by parts formula on the law of the modulus of the Brownian bridge $BB=(BB_t)_{0 \leq t \leq 1}$ from $0$ to $0$ in use of methods from white noise analysis and Dirichlet form theory. Additionally to the usual drift term, this formula contains a distribution which is constructed in the space of Hida distributions by means of a Wick product with Donsker's delta (which correlates with the local time of $|BB|$ at zero). This additional distribution corresponds to the reflection at zero caused by the modulus. \\

\end{abstract}

\section{Introduction} \label{secintro}

%The aim of this paper is to prove an integration by parts formula on the law of the modulus of the Brownian bridge $BB=(BB_t)_{0 \leq t \leq 1}$ from $0$ to $0$. \\ 
Integration by parts formulas are of particular importance with regard to the characterization of stochastic processes constructed in use of Dirichlet form methods in terms of a Fukushima decomposition. This construction and decomposition of stochastic processes provides a useful tool in order to identify solutions to SDEs and SPDEs. In finite dimensions the corresponding generator has been characterized in very general settings with various boundary behaviors. In the case of infinite dimensional processes with a non-trivial boundary behavior the results are less numerous, but of special interest with regard to infinite dimensional stochastic reflection problems. E.g. L. Zambotti proved in \cite{Zam02} an integration by parts formula on the law of the Bessel bridge of dimension $\delta=3$ on the set of continuous, non-negative paths and solved an associated infinite dimensional stochastic reflection problem of Nualart-Pardoux type. Results for $\delta >3$ can be found in the follow-up paper \cite{Zam03}. In \cite{RZZ12} this strategy was formulated in a much more general setting in terms of general Gaussian measures on Hilbert spaces, convex sets and BV functions in a Gelfand triple. Furthermore, a general integration by parts formula for Wiener measures restricted to bounded open subsets $\Omega$ of $\mathbb{R}^d$ with sufficiently regular boundary can be found in \cite{Har06}. In particular, this result generalizes the formula stated in \cite{Zam02}. In \cite{Zam05} an integration by parts formula on the law of the reflecting Brownian motion has been shown. \\
Our considerations are motivated by the aim to prove an infinite dimensional It{\^o}-Tanaka formula for the modulus $(|u_t|)_{t \geq 0}$ of the solution $(u_t)_{t \geq 0}$ of the stochastic heat equation with Dirichlet boundary conditions as well as convergence of finite dimensional approximations. The law of the Brownian bridge $BB=(BB_t)_{0 \leq t \leq 1}$ from $0$ to $0$ is the invariant measure for the solution of the stochastic heat equation on $(0,1)$ with Dirichlet boundary conditions. Consequently, the law of the modulus of $BB$ is the invariant measure for $v:=(v_t)_{t \geq 0}$ with $v_t:=|u_t|$, $t \geq 0$. A natural idea is to identify $(v_t)_{t \geq 0}$ as the weak solution of an SPDE with reflection (which is different from the SPDE of Nualart-Pardoux type). This corresponds to an infinite dimensional It{\^o}-Tanaka formula for $(u_t)_{t \geq 0}$. Related results in use of Malliavin calculus and regularization can be found in \cite{GNT05} and \cite{Zam06}. In order to investigate this problem in terms of Dirichlet form techniques an integration by parts formula on the law of the modulus of the Brownian bridge is essential. The Dirichlet form method has the important advantage that it allows to investigate additionally convergence of finite dimensional approximations in terms of Mosco convergence. The approximation by so-called sticky reflected distorted Brownian motions as constructed in \cite{FGV14} seems likely due to the convergence result in \cite{DGZ05}. Heuristically, our integration by parts formula (see Theorem \ref{thmmain} and the regularization in Remark \ref{remarkreg}) implies that $v=(v_t)_{t \geq 0}$ is a weak solution of the SPDE
\[
 \frac{\partial v}{\partial t}= \frac{1}{2} \frac{\partial^2 v}{\partial x^2} + \frac{\partial^2 W}{\partial t \partial x} - \Big( : \Big(\frac{\partial v}{\partial x}\Big)^2: -1 \Big) \frac{\partial}{\partial t} l_t^{0,v} \]
with Dirichlet boundary conditions on $(0,1)$, where $\frac{\partial^2 W}{\partial t \partial x}$ denotes space-time white noise, $: \big(\frac{\partial v}{\partial x}\big)^2:$ means a renormalization (Wick ordering) and $(l_t^{0,v})_{t \geq 0}=(l_t^{0,v}(x))_{t \geq 0,~ x \in (0,1)}$ is a family of local times of $v$ at $0$. The last term of the equation is ill-defined and represents a reflection term which is caused by the modulus. We would like to emphasize the similarity to the one dimensional KPZ equation.\\

The Bessel bridge is a $[0, \infty)$-valued process due to a repulsive drift term and for this reason this reflection is regular in a certain sense, whereas the reflection caused by the modulus is singular and can only be expressed in terms of local times. For this reason, additional difficulties occur. \\
The integration by parts formula on the law of the reflecting Brownian motion in \cite{Zam05} states that for a suitable class of Fr{\'e}chet differentiable functions $F$ on $L^2(0,1)$ and $h \in C_c^2(0,1)$ holds
\begin{align} \label{IBPZam} \mathbb{E} \big( \partial_h F(|B|) \big)=-\mathbb{E} \big( F(|B|) \int_0^1 h''_t ~|B_t| ~dt \big) + \mathbb{E}\big( F(|B|) \int_0^1 h_t :\dot{B}_t^2: dl_t^0 \big),
\end{align}
where $(l_t^0)_{t \geq 0}$ is the right local time of $(|B_t|)_{t \geq 0}$ at zero (which equals two times the right local time of $(B_t)_{t \geq 0}$ at zero) , i.e., it holds a.s.
\[ l_t^0= \lim_{\varepsilon \rightarrow 0} \frac{1}{\varepsilon} \int_0^t \mathbbm{1}_{[0,\varepsilon)}(|B_s|)~ds,
\] 
and the reflection term is defined by
\[ \mathbb{E}\big( F(|B|) \int_0^1 h_t :\dot{B}_t^2: dl_t^0 \big):=\lim_{\varepsilon \rightarrow 0} ~\mathbb{E}\big( F(|B|) \int_0^1 h_t :\dot{B}_{\varepsilon,t}^2: dl_t^0 \big) \]
for a pathwise regularization $(B_{\varepsilon,t})_{t \geq 0}$ of the Brownian motion and $(:\dot{B}_{\varepsilon,t}^2:)_{t \geq 0}$ the renormalization of the square of its derivative (similar to the definitions for the case of the Brownian bridge given in Section \ref{seclimit}). The disadvantage of this result is the fact that the expression $\int_0^1 h_t :\dot{B}_{t}^2: dl_t^0$ is not constructed explicitly. In the following, we use the main ideas of the first part of \cite{Zam05}, methods from white noise analysis and the theory of Dirichlet forms in order to prove an analogous result for the modulus of the Brownian bridge with the essential difference that the corresponding reflection term is not only given as a limit of expectations, but even in terms of a distribution which is constructed by an integral in the space of Hida distributions. This integral of Hida distributions uniquely determines an element in the dual of the Sobolev space $W^{1,2}(L^2(0,1);\mu^{BB})$, where $\mu^{BB}$ denotes the law of the Brownian bridge on $L^2(0,1)$. We would like to remark that our methods can easily be modified in order to improve the results of \cite{Zam05} correspondingly. \\

Let $q_t:=\mathbbm{1}_{[0,t)} - t ~\mathbbm{1}_{[0,1)}$ for $0 < t \leq 1$, $q_0:=0$ and define for $\eta \in L:=L^2(0,1)$ and $0 \leq t \leq 1$
\begin{align} \label{defH} (H\eta)_t:=H^{\eta}_t:= \int_0^1 q_t(s) \eta_s ~ds= \int_0^t \eta_s~ds - t \int_0^1 \eta_s ~ds.
\end{align}
Moreover, define
\[ \mathcal{F}C_b^{\infty}(L):=\big\{ g\big( (\eta_1,\cdot)_{\scriptstyle{L}},\dots, (\eta_k,\cdot)_{\scriptstyle{L}} \big)\big|~k \in \mathbb{N},~g \in C_b^{\infty}(\mathbb{R}^k),~\eta_i \in L,~i=1,\dots,k \big\} \]
%\[ \exp(L):= \text{span}~\{  \sin((\eta,\cdot)_{\scriptstyle{L^2}}), \cos((\eta,\cdot)_{\scriptstyle{L^2}})~|~\eta \in L \} \]
and
\[ \exp(C^{\infty}):=\text{span}~\big\{  \sin((\eta,\cdot)_{\scriptstyle{L}}), \cos((\eta,\cdot)_{\scriptstyle{L}})~\big|~\eta \in C^{\infty}[0,1] \big\}. \]\\

In use of this definitions, Donsker's delta $\delta_0(|BB_t|)$ (see Definition \ref{defdelta}) and the so-called Wick product of Hida distributions (denoted by $\diamond$; see Definition \ref{defWick}), our main result states as follows:

\begin{theorem} \label{thmmain}
Let $h \in C_c^2(0,1)$ and $F \in \mathcal{F}C_b^{\infty}(L)$. Then, it holds
\begin{align} \mathbb{E}\big( \partial_{h} F(|BB|)\big)=- \mathbb{E} \big( F(|BB|) &\int_0^1 h''_t~ |BB_t|~dt \big) \label{eqmain} \\
&+ \Big\langle \! \! \Big\langle F(|BB|), 2\int_0^1 h_t ~\Gamma_t \diamond \delta_0 (|BB_t|)~dt \Big\rangle \! \! \Big\rangle, \notag 
\end{align}
where $\Gamma_t$, $0 < t <1$, is the Hida distribution with $S$-transform given by 
\[ S(\Gamma_t)(\varphi)=\lambda(\dot{H}^{\varphi}_{t},-H^{\varphi}_t,t)
\]
and
$\int_0^1 h_t ~\Gamma_t \diamond \delta_0 (|BB_t|)~dt$ is a Hida distribution-valued integral with $S$-transform given by
\[ S\Big(\int_0^1 h_t ~\Gamma_t \diamond \delta_0 (|BB_t|)~dt\Big)(\varphi)=\int_0^1 h_t~\frac{\lambda(\dot{H}^{\varphi}_{t},-H^{\varphi}_t,t)}{\sqrt{2 \pi (t-t^2)}} \exp\Big(-\frac{(H^{\varphi}_t)^2}{2(t-t^2)} \Big)~dt \]
for $\varphi \in \mathcal{S}(\mathbb{R})$ with
\[ \lambda(x,y,t)=\Big(x + \frac{1}{2}~ y ~\frac{1-2t}{t -t^2}~ \Big)^2 - \frac{1}{4} \frac{1}{t -t^2} \quad \text{for } x,y \in \mathbb{R},~0 <t <1.\]
The expression 
\[ \Big\langle \! \! \Big\langle G(BB), 2\int_0^1 h_t ~\Gamma_t \diamond \delta_0 (|BB_t|)~dt \Big\rangle \! \! \Big\rangle \] 
is well-defined for $G \in \exp(C^{\infty})$ via the $S$-transform and yields a continuous linear functional on $\exp(C^{\infty})$ with respect to the $W^{1,2}(L;\mu^{BB})$-norm. Furthermore, $\exp(C^{\infty})$ is dense in the Sobolev space $W^{1,2}(L;\mu^{BB})$ and thus, $\int_0^1 h_t ~\Gamma_t \diamond \delta_0 (|BB_t|)~dt$ extends uniquely to an element in $(W^{1,2}(L;\mu^{BB}))'$. $F \circ |\cdot| \in W^{1,2}(L;\mu^{BB})$ and hence, the dual paring in (\ref{eqmain}) is well-defined by the continuous extension.
\end{theorem}

\begin{remark}
Let $L_{+}:=\big\{ f \in L ~\big| ~f \geq 0\big\}$, $\mu^{|BB|}:=\mu^{BB} \circ |\cdot|^{-1}$ and define
\[ W^{1,2}(L_{+};\mu^{|BB|}):=\big\{ F \in L^2(L_+;\mu^{|BB|})~\big|~F \circ |\cdot| \in W^{1,2}(L;\mu^{BB}) \big\} \]
with norm $\Vert F \Vert_{\scriptstyle{W^{1,2}(L_{+};\mu^{|BB|})}}:=\Vert F \circ |\cdot|\Vert_{\scriptstyle{ W^{1,2}(L;\mu^{BB})}}$ for $F \in  W^{1,2}(L_{+};\mu^{|BB|})$. $W^{1,2}(L_{+};\mu^{|BB|})$ is a Hilbert space by \cite[Chapter VI, Exercise 1.1]{MR92}. The statement of Theorem \ref{thmmain} implies that $\widetilde{\mathcal{F}C_b^{\infty}}(L) \subset  W^{1,2}(L_{+};\mu^{|BB|})$, where $\widetilde{\mathcal{F}C_b^{\infty}}(L)$ denotes the space of equivalence classes of $\mathcal{F}C_b^{\infty}(L)$ in $L^2(L_+;\mu^{|BB|})$), and 
\begin{align} \label{norm} \Vert F \Vert_{\scriptstyle{ W^{1,2}(L_{+};\mu^{|BB|})}}^2= \Vert F \Vert_{\scriptstyle{L^2(L_+;\mu^{|BB|})}}^2 + \Vert \nabla F \Vert_{\scriptstyle{L^2(L_+;\mu^{|BB|})}}^2 \quad \text{for } F \in \widetilde{\mathcal{F}C_b^{\infty}}(L).\end{align}
Moreover, $\widetilde{\mathcal{F}C_b^{\infty}}(L)$ is dense in $L^2(L_+;\mu^{|BB|})$ by the monotone class theorem. Hence, in view of \cite[Chapter VI, Exercise 1.1]{MR92} the space $W^{1,2}(L_{+};\mu^{|BB|})$ is the domain of the image Dirichlet form of the Gaussian gradient Dirichlet form $(\mathcal{G},W^{1,2}(L;\mu^{BB}))$ defined in Section \ref{sectDF} under the modulus.\\
Let $H^{1,2}(L_{+};\mu^{|BB|})$ be the closure of $\widetilde{\mathcal{F}C_b^{\infty}}(L)$ in $W^{1,2}(L_{+};\mu^{|BB|})$. By (\ref{eqmain}) it holds
\begin{align*} \Big\langle \! \! \Big\langle F(|BB|), 2\int_0^1 h_t ~\Gamma_t \diamond \delta_0 (|BB_t|)~dt \Big\rangle \! \! \Big\rangle =\mathbb{E}\big( \partial_{h} F(|BB|)\big) + \mathbb{E} \big( F(|BB|) &\int_0^1 h''_t~ |BB_t|~dt \big)
\end{align*} 
for $F \in \widetilde{\mathcal{F}C_b^{\infty}}(L)$ and the right hand side is continuous on $\widetilde{\mathcal{F}C_b^{\infty}}(L)$ with respect to the $W^{1,2}(L_{+};\mu^{|BB|})$-norm in use of (\ref{norm}) (see also the proof of Theorem \ref{thmmain}). Therefore, the functional $\int_0^1 h_t ~\Gamma_t \diamond \delta_0 (|BB_t|)~dt$ can also be considered as an element in $\big(H^{1,2}(L_{+};\mu^{|BB|})\big)'$.
\end{remark}

\begin{remark} \label{remarkreg}
$\Gamma_t$, $0 < t <1$, is the limit in the Hida space $(\mathcal{S})'$ of $(\Gamma_{\varepsilon,t})_{\varepsilon >0}$ for $\varepsilon \rightarrow 0$, where 
\begin{align} \label{GammaBB} \Gamma_{\varepsilon,t}:= ~:\dot{BB}^2_{\varepsilon,t}:\Big(\cdot-\langle q_t,\cdot \rangle \frac{q_t}{\Vert q_t \Vert^2_{\scriptstyle{L^2(\mathbb{R})}}}\Big) -1 
\end{align}
with the renormalization $(:\dot{BB}_{\varepsilon,t}^2:)_{t \geq 0}$ of the squared derivative of the pathwise regularization $(BB_{\varepsilon,t})_{t \geq 0}$ of the Brownian bridge given in Section \ref{seclimit}. In use of the Wick product formula stated in Proposition \ref{thmWPF} it follows that for every $0 < t <1$ and $\varepsilon >0$ holds
\[ \Gamma_{\varepsilon,t} \diamond \delta_0(|BB_t|)= \big(:\dot{BB}_{\varepsilon, t}^2:-1 \big) \cdot \delta_0(|BB_t|), \]
where the product on the right hand side is defined pointwisely for the regular distribution $\delta_0(|BB_t|)$ (see (\ref{productG})). For $\varepsilon >0$ such that $\textnormal{supp}(h) \subset (\varepsilon,1-\varepsilon)$ it holds
\[ 2~\int_0^1 h_t ~\Gamma_{\varepsilon,t} \diamond \delta_0 (|BB_t|)~dt=2 ~\int_0^1 h_t ~\big(:\dot{BB}_{\varepsilon,t}^2:-1 \big) \cdot \delta_0 (|BB_t|)~dt=\int_0^1 h_t ~\big(:\dot{BB}_{\varepsilon,t}^2:-1 \big)~ dL_t^0\]
with $L_t^0=2~\int_0^t \delta_0(|BB_s|)~ds$ the right local time of the modulus of Brownian bridge at zero (see Example \ref{exdelta}). Here, the equality holds in the sense of Hida distrubtions (see Remark \ref{remark}). Hence, this approximation is in analogy to the result for the modulus of the Brownian motion in (\ref{IBPZam}). Furthermore, in the limit $\varepsilon \rightarrow 0$ we obtain the distributions stated in Theorem \ref{thmmain} (see Theorem \ref{thmlimit}). 
\end{remark}

Our paper is organzied as follows:

\begin{itemize}
\item Section \ref{sectintro} briefly introduces the main concepts and results on white noise analysis. 
\item In Section \ref{seclimit} the Hida distributions $\Gamma_t$, $0 < t<1$, and $\int_0^1 h_t ~\Gamma_t \diamond \delta_0 (|BB_t|)~dt$ are constructed and approximated by regular distributions.
\item In Section \ref{sectIBP} the connection of the distributions constructed in Section \ref{seclimit} to the integration by parts formula is presented for nice exponential functions which depend on $BB$ (not on $|BB|$).
\item We investigate properties of the Gaussian gradient Dirichlet form on $L^2(L;\mu^{BB})$ in Section \ref{sectDF}, where $\mu^{BB}$ is the law of $BB$ on $L$. In particular, we proof $F \circ |\cdot| \in W^{1,2}(L;\mu^{BB})$ for $F \in \mathcal{F}C_b^{\infty}(L)$.
\item Section \ref{sectmain} contains the proof of Theorem \ref{thmmain}. We use the results of Section \ref{sectDF} in order to extend the results of Section \ref{sectIBP}. Additionally, we calculate the constructed distribution applied to elements in $\exp(C^{\infty})$ explicitly.
\end{itemize}

\section{White noise analysis} \label{sectintro}

\subsection{White noise measure and Hida spaces} \label{sectWN}

In the following, we briefly introduce the main concepts of white noise analysis. For further reading, we refer to \cite{HKPS93}, \cite{Kuo96} and \cite{Oba94}. The results are of particular use in Section \ref{seclimit}.\\

Consider the Gelfand triple
\begin{align} \label{gelfand} \mathcal{S}(\mathbb{R}) \subset L^2(\mathbb{R}) \subset \mathcal{S}'(\mathbb{R}),
\end{align}
where $\mathcal{S}(\mathbb{R})$ denotes the space of smooth, rapidly decreasing functions (which is the projective limit of Hilbert spaces $(\mathcal{H}_p, \Vert \cdot \Vert_p)$, $p \in \mathbb{N}_0$) and $\mathcal{S}'(\mathbb{R})$ is the topological dual of $\mathcal{S}(\mathbb{R})$, called the space of tempered distributions. As usual, in (\ref{gelfand}) we identify $L^2(\mathbb{R})$ with its dual space. The dual pairing on $\mathcal{S}(\mathbb{R}) \times \mathcal{S}'(\mathbb{R})$ is denoted by $\langle \cdot, \cdot \rangle$. We also use the corresponding complexified spaces $\mathcal{S}_{\mathbb{C}}(\mathbb{R})$, $L^2_{\mathbb{C}}(\mathbb{R})$ as well as $\mathcal{S}'_{\mathbb{C}}(\mathbb{R})$ and the corresponding bilinear extension of the dual paring.  \\
The white noise measure $\mu$ is given by the Bochner-Minlos theorem as the unique probability measure on $(\mathcal{S}'(\mathbb{R}), \mathcal{B})$ satisfying
\[ \int_{\mathcal{S}'(\mathbb{R})} \exp(i \langle \varphi,\omega\rangle)~d\mu(\omega)=\exp\Big(-\frac{1}{2}~\int_{\mathbb{R}} \varphi^2~dx\Big) \]
for every $\varphi \in \mathcal{S}(\mathbb{R})$. Consequently, we define $L^2(\mu):=L^2(\mathcal{S}'(\mathbb{R});\mathbb{C};\mu)$. \\
Using the isometry 
\begin{align} \label{isometry} \int_{\mathcal{S}'(\mathbb{R})} | \langle \eta,\omega\rangle|^2~d\mu(\omega)= \int_{\mathbb{R}} |\eta|^2~dx=\Vert \eta \Vert^2_{L^2_{\mathbb{C}}(\mathbb{R})} \quad \text{for all } \eta \in \mathcal{S}_{\mathbb{C}}(\mathbb{R}) 
\end{align}
it is possible to define $\langle f, \cdot \rangle$ as an element of $L^2(\mu)$ for $f \in L^2_{\mathbb{C}}(\mathbb{R})$. In this sense, the dual pairing $\langle \cdot, \cdot \rangle$ extends from $\mathcal{S}_{\mathbb{C}}(\mathbb{R}) \times \mathcal{S}'_{\mathbb{C}}(\mathbb{R})$ to $L^2_{\mathbb{C}}(\mathbb{R}) \times \mathcal{S}'_{\mathbb{C}}(\mathbb{R})$.\\
Hence, the definition $\langle \mathbbm{1}_{[0,t)}, \cdot \rangle$ for $t > 0$ is well-defined as an element of $L^2(\mu)$. In use of the Kolmogorov-$\breve{\text{C}}$ensov-Lo{\`e}ve theorem it follows the existence of a continuous modification with surely continuous paths which are locally $\gamma$-H{\"o}lder continuous for $\gamma \in (0,1/2)$. This modification yields a standard Brownian motion on $(\mathcal{S}'(\mathbb{R}),\mathcal{B},\mu)$. In the following, we fix this modification and denote it by $(B_t)_{t \geq 0}$, where $B_0:=0$. Nevertheless, we may occasionally consider $B_t=\langle \mathbbm{1}_{[0,t)}, \cdot \rangle$ as an element of $L^2(\mu)$.  In this way, for $0 < t \leq 1$
\begin{align} \label{BBWN}
BB_t:=B_t - t B_1 \quad \text{ and } \quad BB_0:=0 
\end{align}
defines a Brownian bridge $BB=(BB_t)_{0 \leq t \leq 1}$ from $0$ to $0$. In particular, it holds for $0 < t \leq 1$
\begin{align} \label{BBWN2} BB_t =\langle \mathbbm{1}_{[0,t)}- t~ \mathbbm{1}_{[0,1)}, \cdot \rangle = \langle q_t, \cdot \rangle \quad \text{ and } BB_0=\langle q_0,\cdot \rangle 
\end{align}
as an element of $L^2(\mu)$ with $q_t=\mathbbm{1}_{[0,t)} - t ~\mathbbm{1}_{[0,1)}$ for $0 < t \leq 1$, $q_0:=0$. \\

The following result is very useful in later use:

\begin{proposition}[Cameron-Martin formula] \label{propCM}
For $\eta \in L^2(\mathbb{R})$ we define
\[ T_{+\eta}: \mathcal{S}'(\mathbb{R}) \rightarrow \mathcal{S}'(\mathbb{R}), \omega \mapsto T_{+\eta}(\omega):=\omega + \eta. \]
Then, the image measure of $\mu$ under $T_{+\eta}$ has the density $\exp(\langle \eta, \cdot \rangle - \frac{1}{2} \langle \eta, \eta \rangle)$ with respect to $\mu$.
\end{proposition}

In analogy to (\ref{gelfand}), the Gelfand triple
\[ (\mathcal{S}) \subset L^2(\mu) \subset (\mathcal{S})' \]
can be constructed, where $(\mathcal{S})$ denotes the space of Hida test functions (which is the projective limit of Hilbert spaces $((\mathcal{H}_p), |\!|\!| \cdot |\!|\!|_p)$, $p \in \mathbb{N}_0$) and $(\mathcal{S})'$ is its dual space, called the space of Hida distributions. The $S$- and $T$-transform of $\Phi \in (\mathcal{S})'$ are defined by
\[ S\Phi(\varphi):= \big\langle \! \! \big\langle :\exp(\langle \varphi, \cdot \rangle):,\Phi \big\rangle \! \! \big\rangle \ \text{ and } \ T\Phi(\varphi):= \big\langle \! \! \big\langle \exp(i\langle \varphi, \cdot \rangle),\Phi \big\rangle \! \! \big\rangle \quad \text{for } \varphi \in \mathcal{S}_{\mathbb{C}}(\mathbb{R})\]
respectively, where $\big\langle \! \! \big\langle \cdot, \cdot \big\rangle \! \! \big\rangle$ denotes the dual paring on $(\mathcal{S}) \times (\mathcal{S})'$ and 
\[:\exp(\langle \varphi, \cdot \rangle):~:=\exp\big(\langle \varphi, \cdot \rangle - \frac{1}{2} \langle \varphi, \varphi \rangle \big) \]
is the Wick exponential. $S$- and $T$-transform are connected by the relation \[ T\Phi(\varphi)=\exp\big(-\frac{1}{2} \langle \varphi,\varphi \rangle \big) S\Phi(i\varphi)\] 
and the generalized expectation is defined by $\mathbb{E}(\Phi):= \big\langle \! \! \big\langle \mathbbm{1}, \Phi \big\rangle \! \! \big\rangle = S\Phi(0)=T\Phi(0)$.\\

The space of Hida distributions can be characterized in terms of so-called $U$-functionals:

\begin{definition}
A function $F:\mathcal{S}(\mathbb{R}) \rightarrow \mathbb{C}$ is a \textit{$U$-functional} if and only if
\begin{itemize}
\item[(i)] $F$ is ray-analytic, i.e., for all $\varphi, \psi \in \mathcal{S}(\mathbb{R})$ the mapping $\mathbb{R} \ni x \mapsto F(\varphi + x \psi) \in \mathbb{C}$ is analytic and therefore, extends to an entire function on $\mathbb{C}$,
\item[(ii)] $F$ is uniformly bounded of exponential order 2, i.e., there exist $0 \leq A,B < \infty$ and $p \in \mathbb{N}_0$ such that
\[ |F(z\varphi)| \leq A \exp(B ~|z|^2 \Vert \varphi \Vert_p^2) \quad \text{for all } z \in \mathbb{C}, ~\varphi \in \mathcal{S}(\mathbb{R}).\]
\end{itemize}
\end{definition}

The following theorems can be found in \cite[Section 4.C]{HKPS93}.

\begin{theorem}[Characterization of Hida distributions] \label{thmHida}
A mapping $F:\mathcal{S}(\mathbb{R}) \rightarrow \mathbb{C}$ is a $U$-functional if and only if $F=S\Phi$ for some $\Phi \in (\mathcal{S})'$ (respectively $F=T\Psi$ for some $\Psi \in (\mathcal{S})'$).
\end{theorem}

Furthermore, we state the following results on the convergence and integration of Hida distributions:

\begin{theorem} \label{thmconv}
Let $(F_n)_{n \in \mathbb{N}}$ be a sequence of $U$-functionals such that
\begin{itemize}
\item[(i)] $(F_n(\varphi))_{n \in \mathbb{N}}$ is a Cauchy sequence in $\mathbb{C}$ for every $\varphi \in \mathcal{S}(\mathbb{R})$,
\item[(ii)] there exist $p \in \mathbb{N}_0$, $N \in \mathbb{N}_0$ and $0 \leq A,B < \infty$ such that 
\[ |F_n(z \varphi)| \leq A \exp(B ~|z|^2 \Vert \varphi \Vert_p^2) \quad \text{for all } \varphi \in \mathcal{S}(\mathbb{R}), ~z \in \mathbb{C}, n \geq N.\]
\end{itemize}
Then, $(S^{-1}F_n)_{n \in \mathbb{N}}$ $\big((T^{-1}F_n)_{n \in \mathbb{N}} \big)$ converges in $(\mathcal{S})'$. Moreover, the $S$-transform $\big(T$-transform$\big)$ of the limit is given by $\lim_{n \rightarrow \infty} F_n(\varphi)$ for $\varphi \in \mathcal{S}(\mathbb{R})$.
\end{theorem}

\begin{theorem} \label{thmint}
Let $(\Omega, \mathcal{F}, m)$ be a measure space and $\Phi: \Omega \rightarrow (\mathcal{S})'$. Set $F(\lambda, \varphi):=S(\Phi(\lambda))(\varphi)$ and assume that
\begin{itemize}
\item[(i)] for every $\varphi \in \mathcal{S}(\mathbb{R})$ the mapping
\[ \Omega \ni \lambda \mapsto F(\lambda, \varphi) \in \mathbb{C} \ \text{ is measurable,}\]
\item[(ii)] there exist $p \in \mathbb{N}_0$ and $A,B:\Omega \rightarrow [0,\infty)$ with $A \in L^1(\Omega;m)$, $B \in L^{\infty}(\Omega;m)$ such that
\[ |F(\lambda, z \varphi)| \leq A(\lambda) \exp(B(\lambda) ~|z|^2 \Vert \varphi \Vert_p^2) \quad \text{for all } \lambda \in \Omega, z \in \mathbb{C}, \varphi \in \mathcal{S}(\mathbb{R}).\]
\end{itemize}
Then, there exists $p' \in \mathbb{N}_0$ such that $\Phi \in L^1(\Omega;(\mathcal{H}_{-p'});m)$. In particular, 
\[ \int_{\Omega} \Phi(\lambda)~dm(\lambda) \in \mathcal{H}_{-p'} \subset (\mathcal{S})' \]
and 
\[ S \big( \int_{\Omega} \Phi(\lambda)~dm(\lambda) \big)(\varphi) =\int_{\Omega} S(\Phi(\lambda))(\varphi)~dm(\lambda) \quad \text{for all } \varphi \in \mathcal{S}(\mathbb{R})\]
(similarly for the $T$-transform).
\end{theorem}

\begin{example}[Donsker's delta] \label{exdelta}
For $a \in \mathbb{R}$ and $0 \neq \eta \in L^2(\mathbb{R})$ we define
\[ \delta_a(\langle \eta, \cdot \rangle):= \frac{1}{2\pi} ~\int_{\mathbb{R}} \exp(i (\langle \eta, \cdot \rangle -a)x)~dx \in (\mathcal{S})'. \]
This definition is motivated by the Fourier representation of the Dirac measure $\delta_a$ and well-defined by Theorem \ref{thmint} with $(\Omega,\mathcal{F},m)=(\mathbb{R},\mathcal{B}(\mathbb{R}),dx)$. \\
Consider again the Brownian motion $(B_t)_{t \geq 0}$ and the Brownian bridge $(BB_t)_{0 \leq t \leq 1}$. It holds for $0 < t < \infty$
\[ S \delta_a(B_t)(\varphi)=\frac{1}{\sqrt{2 \pi t}} \exp\Big(-\frac{1}{2t} \big( \int_0^t \varphi_s~ ds - a \big)^2 \Big) \] 
and for $0 < t < 1$
\[ S \delta_a(BB_t)(\varphi)=\frac{1}{\sqrt{2 \pi (t-t^2)}} \exp\Big(-\frac{1}{2(t-t^2)} \big( \int_0^t \varphi_s~ ds - t \int_0^1 \varphi_s~ds - a \big)^2 \Big) \] 
for $\varphi \in \mathcal{S}(\mathbb{R})$. \\
In use of Theorem \ref{thmconv} it can be shown that $\varphi_{\varepsilon}(\langle \eta, \cdot \rangle -a )$ converges in $(\mathcal{S})'$ to $\delta_a(\langle \eta, \cdot \rangle)$ as $\varepsilon \rightarrow 0$ for $a \in \mathbb{R}$, $0 \neq \eta \in L^2(\mathbb{R})$ and a Dirac sequence $(\varphi_{\varepsilon})_{\varepsilon >0}$, e.g. 
\begin{align} \label{defDS} \varphi_{\varepsilon}(x):= \frac{1}{\sqrt{2\pi \varepsilon}} \exp\big(-\frac{x^2}{2 \varepsilon}~\big) 
\end{align} for $x \in \mathbb{R}$ and $\varepsilon >0$. In this case, it holds by symmetry
\[ \varphi_{\varepsilon}(|BB_t|)=\varphi_{\varepsilon}(|\langle q_t,\cdot \rangle|)=\varphi_{\varepsilon}(\langle q_t,\cdot \rangle)=\varphi_{\varepsilon}(BB_t) \]
as an element in $L^2(\mu)$. It is natural to define $\delta_0(|BB_t|)$, $0 < t < 1$, as the limit of $\varphi_{\varepsilon}(|BB_t|)$ in $(\mathcal{S})'$ for $\varepsilon \rightarrow 0$. In particular, it holds $\delta_0(|BB_t|)=\delta_0(BB_t)$. Moreover,
\begin{align} \label{reprLT} L_t^0=2~\int_0^t \delta_0 (|BB_s|)~ds= 2~\int_0^t \delta_0(BB_s)~ds=2~ l_t^0 \end{align}
for $0 \leq t \leq 1$ in the sense of Hida distributions, where $(L^a_t)_{0 \leq t \leq 1}$ denotes the right local time of the modulus of the Brownian bridge and $(l^a_t)_{0 \leq t \leq 1}$ denotes the central local time of the Brownian bridge at $a \in \mathbb{R}$ in the sense of \cite[Chapter VI, Corollary 1.9]{RY91}, i.e., it holds a.s.
\[ L_t^a=\lim_{\varepsilon \rightarrow 0} \frac{1}{\varepsilon}~\int_0^t \mathbbm{1}_{[a,a+\varepsilon)}(|BB_s|)~ds,\quad l_t^a=\lim_{\varepsilon \rightarrow 0} \frac{1}{2\varepsilon}~\int_0^t \mathbbm{1}_{(a-\varepsilon,a+\varepsilon)}(BB_s)~ds \quad \text{for } 0 \leq t \leq 1, ~a \in \mathbb{R}.\] 
(\ref{reprLT}) follows by a computation in analogy to the approach in the proof of \cite[Proposition 3.1]{Zam05} in use of the occupation times formula
\[ \int_0^t h(s,BB_s) ds=\int_{\mathbb{R}} \int_0^t h(s,a) ~dl_s^a ~da \]
for Borel measurable $h:[0,1] \times \mathbb{R} \rightarrow [0,\infty)$ (see \cite[Chapter VI, Exercise 1.15]{RY91}. Indeed, in this way it possible to show that  the expectation $\mathbb{E}( :\exp(\langle \varphi, \cdot \rangle): l_t^0)$ equals the $S$-transform of the $(\mathcal{S})'$-valued Bochner integral $\int_0^t \delta_0 (BB_s)~ds$ at $\varphi \in \mathcal{S}(\mathbb{R})$. Note that $l_t^0 \in L^2(\mu)$ in view of \cite[Chapter VI, Exercise 2.35]{RY91}. In the case of a Brownian motion the representation of the local time corresponding to (\ref{reprLT}) is also discussed in \cite[Example 13.24]{Kuo96}.
\end{example}

We recapitulate the definition of $\delta_0(|BB_t|)$, $0 < t <1$, given in Example \ref{exdelta}:

\begin{definition} \label{defdelta}
$\delta_0(|BB_t|)$, $0 < t < 1$, is defined as the limit of $\varphi_{\varepsilon}(|BB_t|)$ in $(\mathcal{S})'$ for $\varepsilon \rightarrow 0$, where $\varphi_{\varepsilon}$ is given by (\ref{defDS}).
\end{definition}

Due to the equality $BB_t=\langle q_t, \cdot \rangle$ in $L^2(\mu)$ the statement of the following lemma seems natural. However, $(BB_t)_{0 \leq t \leq 1}$ is only a continuous modification of $(\langle q_t, \cdot \rangle)_{0 \leq t \leq 1}$, i.e., $BB_t(\omega)=\langle q_t, \omega \rangle$ for $\omega \in \mathcal{S}'(\mathbb{R}) \backslash N_t$ and some $\mu$-null set $N_t$ (depending on $t$). Therefore, the proof requires caution.

\begin{lemma} \label{lemma1}
For $h \in C[0,1]$ and $\mu$-a.e. $\omega \in \mathcal{S}'(\mathbb{R})$ it holds
\[ \int_0^1 h_t~BB_t(\omega)~dt = \Big\langle \int_0^1 h_t~q_t~dt, \omega \Big\rangle, \]
where $\int_0^1 h_t~q_t~dt \in L^2(\mathbb{R})$ defined as a Bochner integral.
\end{lemma}

\begin{proof} 
	The sequence $(g^k)_{k \in \mathbb{N}}$ given by
	\[ [0,1] \ni t \mapsto g^k_t:=\sum_{i=1}^k h_{\frac{i}{k}}~q_{\frac{i}{k}}~\mathbbm{1}_{[\frac{i-1}{k},\frac{i}{k}]}(t) \in L^2(\mathbb{R}) \]
	is a Riemann approximation of the continuous function given by $[0,1] \ni t \mapsto h_t~q_t \in L^2(\mathbb{R})$. In particular, Riemann and Bochner integral coincide and we have
	\begin{align} \label{convL2} \int_0^1 g_t^k ~dt \rightarrow \int_0^1 h_t~q_t~dt \quad \text{in } L^2(\mathbb{R}) ~\text{ as } k \rightarrow \infty.\end{align} 
	Moreover, 
	\[ [0,1] \ni t \mapsto g^{k,\omega}_t:=\sum_{i=1}^k h_{\frac{i}{k}}~BB_{\frac{i}{k}}(\omega)~\mathbbm{1}_{[\frac{i-1}{k},\frac{i}{k}]}(t) \in \mathbb{R} \]
	is a Riemann approximation of $[0,1] \ni t \mapsto h_t~BB_t(\omega) \in \mathbb{R}$ for every $\omega \in \mathcal{S}'(\mathbb{R})$, since $(BB_t)_{0 \leq t \leq 1}$ has surely continuous paths. Thus, by possibly switching to a subsequence we get
	\begin{align*} \int_0^1 h_t ~BB_t(\omega) ~dt = \lim_{k \rightarrow \infty} \int_0^1 &g_t^{k,\omega}~dt =\lim_{k \rightarrow \infty} ~ \sum_{i=1}^k h_{\frac{i}{k}}~BB_{\frac{i}{k}}(\omega)~\frac{1}{k} =\lim_{k \rightarrow \infty} ~ \sum_{i=1}^k h_{\frac{i}{k}}~\langle q_{\frac{i}{k}}, \omega \rangle~\frac{1}{k} \\ 
	&=\lim_{k \rightarrow \infty}  \Big\langle \sum_{i=1}^k h_{\frac{i}{k}}~ q_{\frac{i}{k}}~\frac{1}{k} , \omega \Big\rangle =\lim_{k \rightarrow \infty}  \Big\langle \int_0^1 g_t^k~dt , \omega \Big\rangle = \Big\langle \int_0^1 h_t~q_t~dt, \omega \Big\rangle
	\end{align*}
	for $\mu$-a.e. $\omega \in \mathcal{S}'(\mathbb{R})$ in use of (\ref{BBWN2}), the isometry (\ref{isometry}) and (\ref{convL2}).
\end{proof}

%\begin{proof} 
%The sequence $(g^k)_{k \in \mathbb{N}}$ given by
%\[ [0,1] \ni t \mapsto g^k_t:=\sum_{i=1}^k h_{\frac{i}{k}}~q_{\frac{i}{k}}~\mathbbm{1}_{[\frac{i-1}{k},\frac{i}{k}]}(t) \in L^2(\mathbb{R}) \]
%converges in $L^2((0,1);L^2(\mathbb{R}))$ to the function given by $[0,1] \ni t \mapsto h_t~q_t \in L^2(\mathbb{R})$. In particular, 
%\begin{align} \label{convL2} \int_0^1 g_t^k ~dt \rightarrow \int_0^1 h_t~q_t~dt \quad \text{in } L^2(\mathbb{R}) ~\text{ as } k \rightarrow \infty.\end{align} 
%Moreover, 
%\[ [0,1] \ni t \mapsto g^{k,\omega}_t:=\sum_{i=1}^k h_{\frac{i}{k}}~\langle q_{\frac{i}{k}}, \omega \rangle~\mathbbm{1}_{[\frac{i-1}{k},\frac{i}{k}]}(t) \in \mathbb{R} \]
%is a Riemann approximation of $[0,1] \ni t \mapsto h_t~\langle q_t, \omega \rangle=h_t~BB_t(\omega) \in \mathbb{R}$ for $\mu$-a.e. $\omega \in \mathcal{S}'(\mathbb{R})$. Thus, by possibly switching to a subsequence we get
%\begin{align*} \int_0^1 h_t ~\langle q_t, \omega \rangle ~dt = \lim_{k \rightarrow \infty} \int_0^1 &g_t^{k,\omega}~dt =\lim_{k \rightarrow \infty} ~ \sum_{i=1}^k h_{\frac{i}{k}}~\langle q_{\frac{i}{k}}, \omega \rangle~\frac{1}{k} \\ 
%&=\lim_{k \rightarrow \infty}  \Big\langle \sum_{i=1}^k h_{\frac{i}{k}}~ q_{\frac{i}{k}}~\frac{1}{k} , \omega \Big\rangle =\lim_{k \rightarrow \infty}  \Big\langle \int_0^1 g_t^k~dt , \omega \Big\rangle = \Big\langle \int_0^1 h_t~q_t~dt, \omega \Big\rangle
%\end{align*}
%for $\mu$-a.e. $\omega \in \mathcal{S}'(\mathbb{R})$ in use of the isometry (\ref{isometry}) and (\ref{convL2}).
%\end{proof}

\begin{lemma} \label{lemma2}
For $h \in C[0,1]$ and a.e. $s \in \mathbb{R}$ it holds
\[ \int_0^1 h_t~q_t~dt ~(s)= \begin{cases}
\int_0^1 h_t~q_t(s)~dt & \text{for } s \in [0,1] \\
0                      & \text{for } s \in \mathbb{R} \backslash [0,1]
\end{cases}  \]
and the function $[0,1] \ni s \mapsto \int_0^1 h_t~q_t(s)~dt=\int_s^1 h_t ~dt - \int_0^1 h_t~t~dt$ is continuous.
\end{lemma}

\begin{proof}
Let $(g^k)_{k \in \mathbb{N}}$ be the sequence defined in the proof of Lemma \ref{lem1}. It holds for a.e. $s \in [0,1]$ by possibly switching to a subsequence
\begin{align*}
\int_0^1 h_t~q_t~dt ~(s)=\lim_{k \rightarrow \infty} \int_0^1 g_t^k~dt (s)=\lim_{k \rightarrow \infty} \sum_{i=1}^k h_{\frac{i}{k}}~ q_{\frac{i}{k}}(s)~\frac{1}{k} &=\int_0^1 h_t~q_t(s)~dt \\
&= \int_s^1 h_t ~dt - \int_0^1 h_t~t~dt
\end{align*}
and the last expression is continuous in $s$. Similarly, it follows $\int_0^1 h_t~q_t~dt ~(s)=0$ for a.e. $s \in \mathbb{R} \backslash [0,1]$, since $\text{supp}(q_t) \subset [0,1]$ for every $0 \leq t \leq 1$.
\end{proof}

\subsection{The Wick product}

Since the space of $U$-functionals is closed under pointwise multiplication, we can define the following in view of Theorem \ref{thmHida}:

\begin{definition} \label{defWick}
For $\Phi, \Psi \in (\mathcal{S})'$ the \textit{Wick product} $\Phi \diamond \Psi \in (\mathcal{S})'$ is given via the $S$-transform $S(\Phi \diamond \Psi)=S\Phi \cdot S\Psi$. 
\end{definition}

We additionally introduce the space of test functions $\mathcal{G}$ and its dual $\mathcal{G}'$. A similar construction as in the case of Hida distributions yields the Gelfand triple $\mathcal{G} \subset L^2(\mu) \subset \mathcal{G}'$ such that
\begin{align*} (\mathcal{S}) \subset \mathcal{G} \subset L^2(\mu) \subset \mathcal{G}' \subset (\mathcal{S})'.
\end{align*}
These spaces were was first introduced and analyzed in \cite{PT95}. A characterization is given in \cite{GKS97}. We would like to remark that elements in $L^2(\mu)$ of the form 
\begin{align} \label{inG} F=\langle f , \cdot \rangle^2 - \Vert f \Vert_{\scriptstyle{L^2(\mathbb{R})}}^2=\langle f^{\otimes 2}, :\cdot^{\otimes 2}:\rangle \quad \text{for } f \in L^2(\mathbb{R}) \end{align}
are contained in $\mathcal{G}$, where $:\cdot^{\otimes n}\hspace{-0.1cm}:$, $n  \in \mathbb{N}$, denotes the Wick ordering. The elements in $\mathcal{G}'$ are called regular distributions and have the property that each chaos of the generalized chaos decomposition is an element in $L^2(\mu)$. For example, $\delta_a(\langle \eta, \cdot \rangle) \in \mathcal{G}'$ for $a \in \mathbb{R}$ and $0 \neq \eta \in L^2(\mathbb{R})$. Moreover, $\mathcal{G}$ is closed under pointwise multiplication. This mulitplication yields a continuous mapping from $\mathcal{G} \times \mathcal{G}$ to $\mathcal{G}$ and can be extended to a continuous mapping from $\mathcal{G} \times \mathcal{G}'$ to $\mathcal{G}'$ by 
\begin{align} \label{productG}  \big\langle \! \! \big\langle G, F \cdot \Phi \big\rangle \! \! \big\rangle:= \big\langle \! \! \big\langle F \cdot G, \Phi \big\rangle \! \! \big\rangle \quad \text{for } F, G \in \mathcal{G} \text{ and }\Phi \in \mathcal{G}',~\text{see \cite{PT95}}.  \end{align}
In particular, the multiplication of Donsker's delta and an element in $\mathcal{G}$ is well-defined. For a presentation of the concepts of regular distributions of white noise analysis see \cite{GKS99}.\\
The following result yields a useful method in order to determine the $S$-transform of this product. For a proof we refer to \cite[Section 4.4]{Vog10}, see also \cite{GRS14} and \cite{GR15} for a refinement giving a pointwise meaning.

\begin{proposition}[Wick product formula] \label{thmWPF}
Let $F \in \mathcal{G}$, $0 \neq \eta \in L^2(\mathbb{R})$ and $a \in \mathbb{R}$. Then,
\begin{align}
F \cdot \delta_a(\langle \eta, \cdot \rangle)=F\Big(\cdot + \big(a - \big\langle \frac{\eta}{\Vert \eta \Vert_{\scriptstyle{L^2(\mathbb{R})}}},~ \cdot ~\big\rangle \big) \frac{\eta}{\Vert \eta \Vert_{\scriptstyle{L^2(\mathbb{R})}}} \Big) \diamond \delta_a(\langle \eta, \cdot \rangle) \in \mathcal{G}'.
\end{align}
\end{proposition}

\section{Construction and approximation of the distributions} \label{seclimit}

Let $\rho \in C_c^{\infty}(-1,1)$ such that $\rho \geq 0$, $\int_{-1}^1 \rho~dx=1$ and $\rho(x)=\rho(-x)$ for every $x \in \mathbb{R}$. Then, the family $(\rho_{\varepsilon})_{\varepsilon >0}$ of smooth, symmetric mollifiers is defined by 
\[ \rho_{\varepsilon}(x) := \frac{1}{\varepsilon}~\rho(\frac{x}{\varepsilon}) \quad \text{for } x \in \mathbb{R} \text{ and } \varepsilon >0.\]
For $g \in L^2(0,1)$ and $0 \leq t \leq 1$ define
\[ g_{\varepsilon,t} := (\rho_{\varepsilon} * g)_t:=\int_0^t \rho_{\varepsilon}(s-t) g_s~ds \]
with derivative  
\[ \dot{g}_{\varepsilon,t}=\frac{d}{dt} g_{\varepsilon,t}=-\int_0^t \rho'_{\varepsilon}(s-t) g_s~ds. \]
Let $BB=(BB_t)_{0 \leq t \leq 1}$ be the Brownian bridge from $0$ to $0$ as defined in (\ref{BBWN}). In particular, $BB_t=\langle q_t,\cdot \rangle$ in $L^2(\mu)$ for $0 \leq t \leq 1$. Then, the pathwise smoothing of $BB$ is given by
\begin{align}
BB_{\varepsilon,t}=(\rho_{\varepsilon} * BB)_t = \int_0^t \rho_{\varepsilon}(s-t) BB_s ~ds
\end{align}
for $0 \leq t \leq 1$ and consequently, the pathwise derivative of $(BB_{\varepsilon,t})_{0 \leq t \leq 1}$ with respect to $t$ is given by
\begin{align}
\dot{BB}_{\varepsilon,t}:=\frac{d}{dt}BB_{\varepsilon,t}=(-\rho'_{\varepsilon} * BB)_t=- \int_0^1 \rho'_{\varepsilon}(s-t) BB_s ~ds.
\end{align}
Define
\[ :\dot{BB}_{\varepsilon,t}^2: ~:= \dot{BB}_{\varepsilon,t}^2 - \mathbb{E}(\dot{BB}_{\varepsilon,t}^2) \ \text{ and } \ \Gamma_{\varepsilon,t}:= ~:\dot{BB}^2_{\varepsilon,t}: \Big(\cdot-\langle q_t,\cdot \rangle \frac{q_t}{\Vert q_t \Vert^2_{\scriptstyle{L^2(\mathbb{R})}}}\Big) -1 \quad \text{for } 0 < t < 1, ~\varepsilon >0. \]
Note that $\mathbb{E}(\dot{BB}_{\varepsilon,t}^2)=\frac{\Vert \rho \Vert_{\scriptstyle{L^2(\mathbb{R})}}^2}{\varepsilon} - 1$ for $t \in (\varepsilon,1-\varepsilon)$ and hence, $:\dot{BB}_{\varepsilon,t}^2:-1=\dot{BB}_{\varepsilon,t}^2 -\frac{\Vert \rho \Vert_{\scriptstyle{L^2(\mathbb{R})}}^2}{\varepsilon}$ (which is in analogy to the definition of $:\dot{B}^2_{\varepsilon,t}:$ in \cite{Zam05} in the case of a Brownian motion). In particular, it holds
\begin{align*} \dot{BB}_{\varepsilon,t}(\omega)= -\int_0^1 \rho'_{\varepsilon}(s-t) BB_s(\omega) ~ds=\Big\langle -\int_0^1 \rho'_{\varepsilon}(s-t) q_s~ds, \omega \Big\rangle
\end{align*}
for $\mu$-a.e. $\omega \in \mathcal{S}'(\mathbb{R})$ by Lemma \ref{lemma1}, where the integral on the right hand side is defined in the sense of a Bochner integral with values in $L^2(\mathbb{R})$. Thus, by (\ref{inG}) it follows $:\dot{BB}_{\varepsilon,t}^2: -1 \in \mathcal{G}$, since
\begin{align*} :\dot{BB}_{\varepsilon,t}^2:(\omega)=\Big\langle \Big(-\int_0^1 \rho'_{\varepsilon}(s-t) q_s~ds\Big)^{\otimes 2}, :\omega^{\otimes 2}: \Big\rangle \quad \text{for } ~\mu\text{-a.e.}~ \omega \in \mathcal{S}'(\mathbb{R}).
\end{align*}

\begin{lemma} \label{lem1}
Let $0 < t <1$, $\varepsilon >0$ and $\varphi \in \mathcal{S}(\mathbb{R})$. Then, 
\[ \Gamma_{\varepsilon,t} \diamond \delta_0(|BB_t|)= \big( :\dot{BB}_{\varepsilon,t}^2: -1 \big) \cdot \delta_0(|BB_t|) \in \mathcal{G}'. \]
Furthermore, for $t \in (\varepsilon, 1- \varepsilon)$ the $S$-transform of $\Gamma_{\varepsilon,t}$ is given by 
\[ S(\Gamma_{\varepsilon,t})(\varphi)=\lambda(\dot{H}^{\varphi}_{\varepsilon,t},-H^{\varphi}_t,t), \]
where $H^{\varphi}_t:=\int_0^1 q_t(s)\varphi_s~ds$ (see (\ref{defH})) and
\[ \lambda(x,y,t):=\Big(x + \frac{1}{2}~ y ~\frac{1-2t}{t -t^2} ~\Big)^2 - \frac{1}{4} \frac{1}{t -t^2} \quad \text{for } x,y \in \mathbb{R},~0 <t <1.\]
\end{lemma}

\begin{proof}
Define 
\[F:=~:\dot{BB}_{\varepsilon,t}^2: -1 \in \mathcal{G} \quad \text{for } 0<t<1, ~\varepsilon >0.\]
The first part of the assertion follows by Proposition \ref{thmWPF}. \\
It remains to determine the $S$-transforms of $\Gamma_{\varepsilon,t}=F(\cdot  - \langle q_t, \cdot \rangle \frac{q_t}{\Vert q_t \Vert_{\scriptstyle{L^2(\mathbb{R})}}^2})$ for $t \in (\varepsilon,1-\varepsilon)$. In this case, it holds $F=\dot{BB}_{\varepsilon,t}^2 - \frac{\Vert \rho \Vert_{\scriptstyle{L^2(\mathbb{R})}}^2}{\varepsilon}$ and calculation yields for $\mu$-a.e. $\omega \in \mathcal{S}'(\mathbb{R})$
\begin{align*}
&F\Big(\omega  - \langle q_t, \omega \rangle \frac{q_t}{\Vert q_t \Vert_{\scriptstyle{L^2(\mathbb{R}})}^2}\Big) \\
&=\Big\langle - \int_0^1 \rho'_{\varepsilon}(s-t) q_s~ds, \omega - \langle q_t, \omega \rangle \frac{q_t}{\Vert q_t \Vert_{\scriptstyle{L^2(\mathbb{R})}}^2} \Big\rangle^2  - \frac{\Vert \rho \Vert_{\scriptstyle{L^2(\mathbb{R})}}^2}{\varepsilon} \\
&=\Big\langle \Big(- \int_0^1 \rho'_{\varepsilon}(s-t) q_s~ds \Big)^{\otimes 2}, :\omega^{\otimes 2}: \Big\rangle - 2 \gamma(t)~ \Big\langle \Big(- \int_0^1 \rho'_{\varepsilon}(s-t) q_s~ds \Big) \otimes q_t, :\omega^{\otimes 2}:\Big\rangle \\
&~+ \gamma(t)^2 ~\big\langle q_t^{\otimes 2}, :\omega^{\otimes 2}: \big\rangle - (t- t ^2)~\gamma(t)^2 -1,
\end{align*}
where 
\[\gamma(t):=\frac{1}{\Vert q_t \Vert_{\scriptstyle{L^2(\mathbb{R})}}^2} \int_0^1 -\int_0^1 \rho_{\varepsilon}(s-t) q_s(u) ds~q_t(u)~du=\frac{1}{2} \frac{1-2t}{t-t^2}\]
with $\Vert q_t \Vert_{\scriptstyle{L^2(\mathbb{R})}}^2=t-t^2$. Thus, for $\varphi \in \mathcal{S}(\mathbb{R})$ holds in use of Lemma \ref{lemma2}
\begin{align*}
S\Big(F \big(\cdot  - \langle q_t, \cdot \rangle \frac{q_t}{\Vert q_t \Vert_{\scriptstyle{L^2(\mathbb{R})}}^2} \big)\Big)(\varphi)&=\Big( \int_0^1 -\int_0^1 \rho'_{\varepsilon}(t-s)q_s(u) ~ds~ \varphi_u du \Big)^2 \\
&~- 2 \gamma(t) ~\Big( \int_0^1 - \int_0^1 \rho'_{\varepsilon}(t-s) q_s(u)~ds~ \varphi_u~du \Big) \Big( \int_0^1 q_t(u) \varphi_u~du\Big)\\
&~+ \gamma(t)^2 ~\Big(\int_0^1 q_t(u) \varphi_u~du \Big)^2 -(t- t ^2)~\gamma(t)^2 -1 \\
&= (\dot{H}^{\varphi}_{\varepsilon,t})^2 -2  \gamma(t) \dot{H}^{\varphi}_{\varepsilon,t}~H^{\varphi}_t + \gamma(t)^2~(H^{\varphi}_t)^2 -(t- t ^2)~\gamma(t)^2 -1 \\
&= \lambda(\dot{H}^{\varphi}_{\varepsilon,t},-H^{\varphi}_t,t).
\end{align*}
\end{proof}

\begin{theorem} \label{thmlimit}
For each $0 < t <1$ $(\Gamma_{\varepsilon,t})_{\varepsilon >0}$ with $\Gamma_{\varepsilon,t} \in \mathcal{G}'$ for each $\varepsilon >0$ converges in $(\mathcal{S})'$ as $\varepsilon \rightarrow 0$ to the Hida distribution $\Gamma_t$ with $S$-transform given by
\[ S(\Gamma_t)(\varphi)=\lambda(\dot{H}^{\varphi}_{t},-H^{\varphi}_t,t) \quad \text{for } \varphi \in \mathcal{S}'(\mathbb{R}).\]
Let $h \in C_c^2(0,1)$ and $\varepsilon >0$ such that $\textnormal{supp}(h) \subset (\varepsilon,1-\varepsilon)$. Then, it exists the integral
\[ \Phi_{\varepsilon,h}:=\int_0^1 h_t ~\Gamma_{\varepsilon,t} \diamond \delta_0 (|BB_t|)~dt \]
in $(\mathcal{S})'$ with $S$-transform
\[ S(\Phi_{\varepsilon,h})(\varphi)=\int_0^1 h_t~\frac{\lambda(\dot{H}^{\varphi}_{\varepsilon,t},-H^{\varphi}_t,t)}{\sqrt{2 \pi (t-t^2)}} \exp\Big(-\frac{(H^{\varphi}_t)^2}{2(t-t^2)} \Big)~dt \quad \text{for } \varphi \in \mathcal{S}(\mathbb{R}), \]
where $H$, $\lambda$ are given as in Lemma \ref{lem1}. Moreover, the integral
\[ \Phi_{h}:=\int_0^1 h_t ~\Gamma_t \diamond \delta_0 (|BB_t|)~dt \]
exists in $(\mathcal{S})'$ with $S$-transform given by
\[ S(\Phi_h)(\varphi)=\int_0^1 h_t~\frac{\lambda(\dot{H}^{\varphi}_{t},-H^{\varphi}_t,t)}{\sqrt{2 \pi (t-t^2)}} \exp\Big(-\frac{(H^{\varphi}_t)^2}{2(t-t^2)} \Big)~dt \quad \text{for } \varphi \in \mathcal{S}(\mathbb{R})\]
and $(\Phi_{\varepsilon,h})_{\varepsilon >0}$ converges in $(\mathcal{S})'$ to  $\Phi_h$ as $\varepsilon \rightarrow 0$. 
\end{theorem}

\begin{proof}
Note that $\Vert \varphi \Vert_{\infty} \leq C \Vert \varphi \Vert_{p}$ for some $p \in \mathbb{N}$, a constant $0 < C < \infty$ and every $\varphi \in \mathcal{S}(\mathbb{R})$, since $\Vert \cdot \Vert_{\infty}$ is continuous on $\mathcal{S}(\mathbb{R})$. Thus, it is sufficient to estimate the involved $S$-transforms in use of the norm $\Vert \cdot \Vert_{\infty}$. \\
For the first part of the claim, we apply Theorem \ref{thmconv}. For $\varphi \in \mathcal{S}(\mathbb{R})$, $0 < t <1$ and $\varepsilon$ small enough such that $t \in (\varepsilon,1-\varepsilon)$ we conclude by Lemma \ref{lem1}, the continuity of $\lambda$ and the definition of $H$ and its regularization that
\[ S(\Gamma_{\varepsilon,t})(\varphi)=\lambda(\dot{H}^{\varphi}_{\varepsilon,t},-H^{\varphi}_t,t) \rightarrow \lambda(\dot{H}^{\varphi}_{t},-H^{\varphi}_t,t) \quad \text{as } \varepsilon \rightarrow 0.\]
Furthermore, for $z \in \mathbb{C}$ holds
\[  \left|\lambda(\dot{H}^{z \varphi}_{\varepsilon,t},-H^{z \varphi}_t,t) \right| = \left|\lambda(z\dot{H}^{ \varphi}_{\varepsilon,t},-zH^{\varphi}_t,t)\right|  \leq \Big(8 + 2 t^2 \frac{(1-2t)^2}{(t-t^2)^2} + \frac{1}{4} \frac{1}{t-t^2} \Big)~\exp \big(|z|^2~\Vert \varphi \Vert_{\infty}^2 \big),\]
where we used $|H_t^{\varphi}| \leq 2t ~\Vert \varphi \Vert_{\infty}$, $|\dot{H}_{\varepsilon,t}^{\varphi}| \leq 2 \Vert \varphi \Vert_{\infty}$ and the inequalities $(a+b)^2 \leq 2a^2+2b^2$, $a,b \in \mathbb{R}$, $\max \{1,c\} \leq \exp(c)$, $c \in (0,\infty)$. \\
For the second assertion, we apply Theorem \ref{thmint} with $(\Omega,\mathcal{F},m)=((0,1), \mathcal{B}(0,1),dx)$. Clearly, in use of Lemma \ref{lem1} the mapping
\[ (0,1) \ni t \mapsto S(h_t ~\Gamma_{\varepsilon,t} \diamond \delta_0(|BB_t|))(\varphi)= h_t~\frac{\lambda(\dot{H}^{\varphi}_{\varepsilon,t},-H^{\varphi}_t,t)}{\sqrt{2 \pi (t-t^2)}} \exp\Big(-\frac{(H^{\varphi}_t)^2}{2(t-t^2)} \Big)  \]
is measurable for every $\varphi \in \mathcal{S}(\mathbb{R})$. Moreover, for $z=z_1 + i z_2 \in \mathbb{C}$, $z_1,z_2 \in \mathbb{R}$, holds
\begin{align*} &\left|  h_t~\frac{\lambda(\dot{H}^{z \varphi}_{\varepsilon,t},-H^{z \varphi}_t,t)}{\sqrt{2 \pi (t-t^2)}} \exp\Big(-\frac{(H^{z \varphi}_t)^2}{2(t-t^2)} \Big) \right| \\
= &\left|h_t~\frac{\lambda(z\dot{H}^{ \varphi}_{\varepsilon,t},-zH^{\varphi}_t,t)}{\sqrt{2 \pi (t-t^2)}} \right| \left|\exp\Big(-\frac{z^2~(H^{\varphi}_t)^2}{2(t-t^2)} \Big) \right| 
%&=\left|h_t~\frac{\lambda(z\dot{H}^{ \varphi}_{\varepsilon,t},-zH^{\varphi}_t,t)}{\sqrt{2 \pi (t-t^2)}} \right| \exp\Big((z_2^2 -z_1^2)\frac{(H^{\varphi}_t)^2}{2(t-t^2)} \Big) \\
\leq \left|h_t~\frac{\lambda(z\dot{H}^{ \varphi}_{\varepsilon,t},-zH^{\varphi}_t,t)}{\sqrt{2 \pi (t-t^2)}} \right| \exp\Big(|z|^2~\frac{(H^{\varphi}_t)^2}{2(t-t^2)} \Big) \\
&\hspace{7.8cm} \leq A(t) \exp\Big( |z|^2 ~\Vert \varphi \Vert_{\infty}^2 B(t) \Big)
\end{align*}
with
\[ A(t):=|h_t| ~\frac{1}{\sqrt{2 \pi (t-t^2)}}~\big(8 + 2 t^2 \frac{(1-2t)^2}{(t-t^2)^2} + \frac{1}{4} \frac{1}{t-t^2}\big) \quad \text{ and } \quad B(t):=\frac{2t^2}{t-t^2} + 1 \]
similarly as above. Since $h$ has compact support in $(0,1)$, we can assume that $t \in (\delta,1-\delta)$ for some $\delta >0$ and adapt $A$, $B$ such that $A \in L^1(0,1)$ and $B \in L^{\infty}(0,1)$. Note that the estimate is independent of $\varepsilon$ and holds also for the limit case with $\Gamma_{\varepsilon,t}$ replaced by $\Gamma_t$, since $\dot{H}_t^{\varphi}=\varphi_t -\int_0^1 \varphi_s ~ds$ fufills also the estimate $|\dot{H}_t^{\varphi}| \leq 2 \Vert \varphi \Vert_{\infty}$.\\
For the last assertion, we use again Theorem \ref{thmconv}. It holds
\[ S(\Phi_{\varepsilon,h})(\varphi)=\int_0^1 h_t~\frac{\lambda(\dot{H}^{\varphi}_{\varepsilon,t},-H^{\varphi}_t,t)}{\sqrt{2 \pi (t-t^2)}} \exp\Big(-\frac{(H^{\varphi}_t)^2}{2(t-t^2)} \Big)~dt \]
for $\varphi \in \mathcal{S}(\mathbb{R})$ and $\varepsilon >0$. By the previous part of the proof, for every $0 < t <1$ the modulus of the integrand is bounded by $A(t) \exp(\Vert \varphi \Vert_{\infty} B(t))$, which defines an element of $L^1(0,1)$, since $A \in L^1(0,1)$ and $B \in L^{\infty}(0,1)$. Moreover, $\dot{H}^{\varphi}_{\varepsilon,t} \rightarrow \dot{H}^{\varphi}_{t}$ as $\varepsilon \rightarrow 0$ for every $0 < t < 1$. Hence, by Lebesgue dominated convergence follows
\[ S(\Phi_{\varepsilon,h})(\varphi) \rightarrow \int_0^1 h_t~\frac{\lambda(\dot{H}^{\varphi}_{t},-H^{\varphi}_t,t)}{\sqrt{2 \pi (t-t^2)}} \exp\Big(-\frac{(H^{\varphi}_t)^2}{2(t-t^2)} \Big)~dt \quad \text{as } \varepsilon \rightarrow 0. \]
Finally, the same bound yields the estimate
\[ | S(\Phi_{\varepsilon,h})(z\varphi)| \leq \Vert A \Vert_{\scriptstyle{L^1(0,1)}}~ \exp(|z|^2 \Vert \varphi \Vert_{\infty} \Vert B \Vert_{\scriptstyle{L^{\infty}(0,1)}}) \quad \text{for } z \in \mathbb{C}.\]
\end{proof}

\begin{remark} \label{remark}
Note that in Theorem \ref{thmlimit} the involved $S$-transforms at $\varphi \in \mathcal{S}(\mathbb{R})$ only depend on the values of the test function on $[0,1]$.\\ Moreover,
\[\Phi_{\varepsilon,h}=\int_0^1 h_t ~(:\dot{BB}_{\varepsilon,t}^2: -1) \cdot \delta_0 (|BB_t|)~dt =\int_0^1 h_t ~(:\dot{BB}_{\varepsilon,t}^2: -1) ~dl_t^0\]
for $\varepsilon >0$ such that $\textnormal{supp}(h) \subset (\varepsilon,1-\varepsilon)$ with $(l_t^0)_{0 \leq t \leq 1}$ the central local time of the modulus of the Brownian bridge (as introduced in Example \ref{exdelta}). This equality holds in the space $(\mathcal{S})'$, where 
\[ \mathcal{S}(\mathbb{R}) \ni \varphi \mapsto \mathbb{E} \big( :\exp(\langle \varphi, \cdot \rangle):~\int_0^1 h_t ~(:\dot{BB}_{\varepsilon,t}^2: -1) ~dl_t^0 \big) \]
defines a Hida distribution, and as explained in Example \ref{exdelta} the proof of \cite[Proposition 3.1]{Zam05} shows that this expectation equals $S(\Phi_{\varepsilon,h})(\varphi)$ for $\varepsilon >0$ such that $\textnormal{supp}(h) \subset (\varepsilon,1-\varepsilon)$.  
\end{remark}

\section{A first step towards the integration by parts formula} \label{sectIBP}

In this section, we use It{\^o}'s formula in order to connect the distribution constructed in Section \ref{seclimit} with the expectation of the directional derivative in the directions $h \cdot \text{sgn}(BB)$, $h \in C_c^2(0,1)$, for exponential functions depending on the Brownian bridge $BB$. Note that in general the above directions are not from the Cameron-Martin space of the Brownian bridge. Therefore, the corresponding divergence on the Wiener space can not be computed with the classical theory of the Malliavin calculus. We apply the ideas of \cite[Section 3]{Zam05} to $BB$ which causes additional problems due to the more complicated semimartingale structure.\\

Define the covariance operator $Q:L \rightarrow L$ by 
\begin{align} \label{covariance} (Q\eta)_t:=\int_0^1 (t \wedge s - ts)~ \eta_s~ds \quad \text{for } \eta \in L=L^2(0,1) \end{align}
and moreover, $K:L \rightarrow L, \eta \mapsto K\eta$ by $(K\eta)_t:=\int_t^1 \eta_s~ ds - \int_0^1 s~\eta_s~ds$ for $0 \leq t \leq 1$ (compare to Lemma \ref{lemma2} for the special case $\eta \in C[0,1]$).\\

Let $h \in C_c^2(0,1)$, $\eta \in C[0,1]$ and $\Phi_{\eta}:= \exp(-\frac{1}{2}(Q\eta,\eta)_{\scriptstyle{L}}) ~\exp((\eta,\cdot)_{\scriptstyle{L}})$. We have for $\mu$-a.e. $\omega \in \mathcal{S}'(\mathbb{R})$
\begin{align} \partial_{h \cdot \text{sgn}(BB(\omega))} \Phi_{\eta}(BB(\omega))&:= \lim_{\varepsilon \rightarrow 0} \frac{\Phi_{\eta}(BB(\omega)+\varepsilon ~h ~ \text{sgn}(BB(\omega))) - \Phi_{\eta}(BB(\omega))}{\varepsilon} \label{gateaux} \\
&= (\eta,h~ \text{sgn}(BB(\omega)))_{\small{L}}~\Phi_{\eta}(BB(\omega)), \notag
\end{align}
where $\text{sgn}:= \mathbbm{1}_{(0,\infty)} - \mathbbm{1}_{(-\infty,0]}$.
Moreover, we can conclude by Lemma \ref{lemma1} and Lemma \ref{lemma2} that $(\eta,BB(\omega))_{\scriptstyle{L}}=\langle K\eta, \omega\rangle$ for $\mu$-a.e. $\omega \in \mathcal{S}'(\mathbb{R})$ by extending $K\eta$ to $\mathbb{R}$ by zero. Furthermore, $(Q\eta,\eta)_{\scriptstyle{L}}=(K\eta,K\eta)_{\scriptstyle{L}}$, since $(K\eta)'=-\eta$ and $\int_0^1 (K\eta)_t~dt=0$. Thus, for $\mu$-a.e. $\omega \in \mathcal{S}'(\mathbb{R})$ holds
\begin{align} \Phi_{\eta}(BB(\omega))=\exp(-\frac{1}{2}(K\eta,K\eta)_{\scriptstyle{L}}) ~\exp(\langle K\eta, \omega \rangle)=~:\exp(\langle K\eta, \omega \rangle):. \label{relation}
\end{align}
Consequently, by Proposition \ref{propCM} and the equality $(HK\eta)_t=(q_t,K\eta)_{\scriptstyle{L}}=(Q\eta)_t$, where $H$ is defined in (\ref{defH}), it follows
\begin{align*} \mathbb{E}\big( \partial_{h \cdot \text{sgn}(BB)} \Phi_{\eta}(BB)\big)&=\mathbb{E}\big(:\exp(\langle K\eta, \cdot \rangle):~\int_0^1 h_t~\eta_t~\text{sgn}(\langle q_t, \cdot \rangle)~dt \big)  \\
&=\int_0^1 h_t~\eta_t~\mathbb{E}\big(\text{sgn}(\langle q_t, \cdot + K\eta \rangle) \big)~dt =\int_0^1 h_t~\eta_t~\mathbb{E}\big(\text{sgn}(BB_t + (Q\eta)_t) \big)~dt.
\end{align*}
Since $\text{sgn}$ is the weak derivative of the modulus, we are interested in expressions of the form $\eta_t~\mathbb{E}\big(\phi'(BB_t + (Q\eta)_t) \big)$. This motivates the following lemma:

\begin{lemma} \label{lem2}
Let $\phi \in C^2_b(\mathbb{R})$, $\eta \in C[0,1]$ and $0 < t <1$. Then,
\begin{align*} \eta_t~\mathbb{E}\big(\phi'(BB_t + (Q\eta)_t) \big)&=-\frac{d^2}{dt^2} \mathbb{E}\big(\phi(BB_t + (Q\eta)_t)\big) + \mathbb{E}\big( \phi''(BB_t + (Q\eta)_t)~\lambda((Q\eta)'_t,BB_t,t) \big) \\
&=-\frac{d^2}{dt^2} \mathbb{E}\big( \Phi_{\eta}(BB)~\phi(BB_t)\big) + \mathbb{E}\big( \phi''(BB_t + (Q\eta)_t)~\lambda((Q\eta)'_t,BB_t,t) \big)
\end{align*}
\end{lemma}

% + \mathbb{E}\big( F_{\eta}(BB)~\phi''(BB_t)~\lambda((Q\eta)'_t,BB_t-(Q\eta)_t,t) \big) 

\begin{proof}
As in the proof of \cite[Lemma 3.5]{Zam05}, by approximation it is enough to consider $\phi \in C_b^4(\mathbb{R})$ and the idea is to use It{\^o}'s formula in order to prove the assertion. By \cite[Chapter IV, Exercise 3.18]{RY91} the Brownian bridge $(BB_t)_{0 \leq t \leq 1}$ defined in (\ref{BBWN}) is a solution to the SDE
\begin{align} dX_t&=-\frac{X_t}{1-t} ~dt + dB_t, \label{SDEBB} \\
X_0&=X_1=0, \notag
\end{align}
where $(B_t)_{t \geq 0}$ is a standard Brownian motion. Hence, in the following we consider a semimartingale $(X_t)_{0 \leq t \leq 1}$ solving (\ref{SDEBB}).\\
Set $H_t:=H(K\eta)_t=(Q\eta)_t$ and define $Y_t:=X_t + H_t$ for $0 \leq t \leq 1$ such that 
\[dY_t=dX_t + dH_t=\Big(H'_t - \frac{Y_t - H_t}{1-t}\Big)~dt + dB_t \ \text{ and } \ \langle Y \rangle_t = \langle B \rangle_t =t. \]
Let $\tilde{\phi} : \mathbb{R}_{>0} \times \mathbb{R} \rightarrow \mathbb{R}, (t,x) \mapsto \tilde{\phi}(t,x)$ be continuously differentiable in the first variable and twice continuously differentiable in the second variable. Then, It{\^o}'s formula yields for $0\leq t \leq 1$
\begin{align} \label{ito}
\tilde{\phi}(t,Y_t)=\tilde{\phi}(0,0) &+ \int_0^t \Big( \frac{\partial \tilde{\phi}}{\partial t}(s, Y_s) + \frac{\partial \tilde{\phi}}{\partial x} (s,Y_s) \Big( H'_s - \frac{Y_s - H_s}{1-s} \Big) + \frac{1}{2} \frac{\partial^2 \tilde{\phi}}{\partial x^2} (s,Y_s) \Big)~ds \\
&+ \int_0^t \frac{\partial \tilde{\phi}}{\partial x} (s,Y_s) ~dB_s. \notag
\end{align}
Therefore,
\[ \frac{d}{dt} \mathbb{E}\big( \phi(Y_t) \big)= \mathbb{E} \Big( \phi'(Y_t)\Big(H'_t - \frac{Y_t - H_t}{1-t} \Big) + \frac{1}{2} \phi''(Y_t) \Big)
\]
and 
\[ \frac{d^2}{dt^2} \mathbb{E}\big( \phi(Y_t) \big)=\frac{d}{dt} \mathbb{E} \big( \phi'(Y_t)\big) ~H'_t +  \mathbb{E} \big( \phi'(Y_t)\big) ~H''_t - \frac{d}{dt} \mathbb{E} \big(\phi'(Y_t)~\frac{Y_t - H_t}{1-t} \big) + \frac{1}{2} \frac{d}{dt} \mathbb{E} \big( \phi''(Y_t) \big). \]
By applying (\ref{ito}) again (in particular to $\tilde{\phi}(t,x):=\phi'(x)~\frac{x-H_t}{1-t}$ for $0 < t <1$) it follows
\begin{align} \notag
- \mathbb{E}\big(\phi'(Y_t)\big) ~H''_t = - \frac{d^2}{dt^2} \mathbb{E}\big(\phi(Y_t)\big) &+ \mathbb{E} \Big( \phi''(Y_t) \Big( (H'_t)^2 + \frac{(Y_t - H_t)^2}{(1-t)^2} - 2~\frac{Y_t - H_t}{1-t}~H'_t - \frac{1}{1-t} \Big) \Big) \\
&+ \mathbb{E}\Big(\phi'''(Y_t) \Big( H'_t - \frac{Y_t - H_t}{1-t} \Big) \Big) + \frac{1}{4} \mathbb{E}\big( \phi''''(Y_t) \big), \label{eq2}
\end{align}
since the first order terms on the right hand side cancel. $X_t$, $0 < t <1$, is $\mathcal{N}(0,t-t^2)$-distributed and we have the integration by parts formulas
\begin{align*}
&(t-t^2) \int_{\mathbb{R}} \psi'(y + H_t)~ \mathcal{N}(0,t-t^2)(dy)= \int_{\mathbb{R}} y ~\psi(y+H_t)~\mathcal{N}(0,t-t^2)(dy),  \\
&\int_{\mathbb{R}} y~\psi'(y + H_t)~ \mathcal{N}(0,t-t^2)(dy)= \int_{\mathbb{R}} \Big(\frac{y^2}{t-t^2} -1\Big) ~\psi(y+H_t)~\mathcal{N}(0,t-t^2)(dy),  \\
&(t-t^2)^2 \int_{\mathbb{R}} \psi''(y + H_t)~ \mathcal{N}(0,t-t^2)(dy)= \int_{\mathbb{R}} (y^2 - (t-t^2)) ~\psi(y+H_t)~\mathcal{N}(0,t-t^2)(dy).
\end{align*}
Thus, we can replace the third and fourth order terms in (\ref{eq2}) by terms of second order. Finally, we obtain the representation
\begin{align*} 
- \mathbb{E}\big(\phi'(X_t+H_t)\big) ~H''_t&=- \mathbb{E}\big(\phi'(Y_t)\big) ~H''_t  \\
=&- \frac{d^2}{dt^2} \mathbb{E}\big(\phi(Y_t)\big) + \mathbb{E} \big( \phi''(Y_t) ~\tilde{\lambda}(H'_t,-(Y_t-H_t),t)\big) \\
=&- \frac{d^2}{dt^2} \mathbb{E}\big(\phi(X_t + H_t)\big) + \mathbb{E} \big( \phi''(X_t+ H_t) \tilde{\lambda}(H'_t,-X_t,t)\big),
\end{align*}
where
\[ \tilde{\lambda}(x,y,t):=x^2 + y^2 \Big( \frac{1}{(1-t)^2} - \frac{1}{(1-t)(t-t^2)} + \frac{1}{4} \frac{1}{(t-t^2)^2} \Big) - x~y \Big(\frac{1}{t-t^2} - \frac{2}{1-t} \Big) - \frac{1}{4}~\frac{1}{t-t^2}. \]
Calculation yields $\tilde{\lambda}=\lambda$. The statement follows by the relations $H_t=(Q\eta)_t$ and $-H''_t=-(Q\eta)''_t=\eta_t$ as well as the Cameron-Martin formula in Proposition \ref{propCM}.
\end{proof}

\begin{proposition} \label{propIBP}
Let $h \in C_c^2(0,1)$ and $\eta \in C^{\infty}[0,1]$. Then, it holds
\[ \mathbb{E}\big( \partial_{h \cdot \text{sgn}(BB)} \Phi_{\eta}(BB)\big)=- \mathbb{E} \big( \Phi_{\eta}(BB) \int_0^1 h''_t~ |BB_t|~dt \big) + \Big\langle \! \! \Big\langle \Phi_{\eta}(BB), 2\int_0^1 h_t ~\Gamma_t \diamond \delta_0 (|BB_t|)~dt \Big\rangle \! \! \Big\rangle.\] 
\end{proposition}

\begin{proof}
For $\phi \in C^2_b(\mathbb{R})$ holds in use of Lemma \ref{lem2}
\begin{align*}
&\int_0^1 h_t~\eta_t~\mathbb{E}\big(\phi'(BB_t + (Q\eta)_t) \big)~dt \\
= &- \mathbb{E}\big(\Phi_{\eta}(BB)~\int_0^1 h''_t~\phi(BB_t)~dt\big) + \int_0^1 h_t~\mathbb{E}\big( \phi''(BB_t + (Q\eta)_t)~\lambda((Q\eta)'_t,BB_t,t) \big)~dt.
\end{align*}
Moreover, let $(\phi_k)_{k \in \mathbb{N}}$ be a $C_b^2(\mathbb{R})$-approximation of the modulus such that $\phi_k \rightarrow |\cdot|$, $\phi'_k \rightarrow \text{sgn}$ and $\phi''_k \rightarrow 2~\delta_0$ as $k \rightarrow \infty$ in $\mathcal{S}'(\mathbb{R})$. Then,
\begin{align*}
&\int_0^1 h_t~\mathbb{E}\big( \phi_k''(BB_t + (Q\eta)_t)~\lambda((Q\eta)'_t,BB_t,t) \big)~dt \\
=&\int_0^1 h_t~\int_{\mathbb{R}} \phi_k''(y + (Q\eta)_t)~\frac{\lambda((Q\eta)'_t,y,t)}{\sqrt{2 \pi (t-t^2)}} \exp\Big(- \frac{y^2}{2(t-t^2)}\Big)~dy~dt \\
%=&\int_0^1 h_t~\int_{\mathbb{R}} \phi_k''(y)~\frac{\lambda((Q\eta)'_t,y-(Q\eta)_t,t)}{\sqrt{2 \pi (t-t^2)}} \exp\Big(- \frac{(y-(Q\eta)_t)^2}{2(t-t^2)}\Big)~dy \\
=& \int_{\mathbb{R}} \phi_k''(y)~\int_0^1 h_t~\frac{\lambda((Q\eta)'_t,y-(Q\eta)_t,t)}{\sqrt{2 \pi (t-t^2)}} \exp\Big(- \frac{(y-(Q\eta)_t)^2}{2(t-t^2)}\Big)~dt~dy \\
\longrightarrow & \quad 2~\int_0^1 h_t~\frac{\lambda((Q\eta)'_t,-(Q\eta)_t,t)}{\sqrt{2 \pi (t-t^2)}} \exp\Big(- \frac{(Q\eta)_t)^2}{2(t-t^2)}\Big) \quad \text{as } k \rightarrow \infty,
\end{align*}
since $\varphi \in \mathcal{S}(\mathbb{R})$ with 
\[ \varphi(y):=\int_0^1 h_t~\frac{\lambda((Q\eta)'_t,y-(Q\eta)_t,t)}{\sqrt{2 \pi (t-t^2)}} \exp\Big(- \frac{(y-(Q\eta)_t)^2}{2(t-t^2)}\Big)~dt \quad \text{for } y \in \mathbb{R}. \] 
Similarly, the convergence of the remaining two terms follows. Hence, using this approximation and Theorem \ref{thmlimit} we can conclude that
\begin{align*}
&\mathbb{E}\big( \partial_{h \cdot \text{sgn}(BB)} \Phi_{\eta}(BB)\big)\\
=&\int_0^1 h_t~\eta_t~\mathbb{E}\big(\text{sgn}(BB_t + (Q\eta)_t) \big)~dt \\
=&- \mathbb{E}\big(\Phi_{\eta}(BB)~\int_0^1 h''_t~|BB_t|~dt\big) + 2\int_0^1 h_t~\frac{\lambda((Q\eta)'_t,-(Q\eta)_t,t)}{\sqrt{2 \pi (t-t^2)}} \exp\Big(- \frac{(Q\eta)_t)^2}{2(t-t^2)}\Big)~dt \\
=&- \mathbb{E}\big(\Phi_{\eta}(BB)~\int_0^1 h''_t~|BB_t|~dt\big) + \Big\langle \! \! \Big\langle \Phi_{\eta}(BB),2 \int_0^1 h_t ~\Gamma_t \diamond \delta_0 (|BB_t|)~dt \Big\rangle \! \! \Big\rangle.
\end{align*}
Here, we used again $H(K\eta)_t=(Q\eta)_t$ and the relation (\ref{relation}). Moreover, we can assume that $K\eta$ is the restriction of an element in $\mathcal{S}(\mathbb{R})$ in view of Remark \ref{remark}.
\end{proof}

\section{Gaussian gradient Dirichlet form with reference measure $\mu^{BB}$} \label{sectDF}

We consider the Gaussian gradient Dirichlet form $(\mathcal{G},W^{1,2}(L;\mu^{BB}))$ on $L^2(L;\mu^{BB})$ and prove by approximation that its domain contains $\mathcal{F}C_b^{\infty}(L)$-functions composed with the modulus. This result will be useful in order to conclude Theorem \ref{thmmain} in use of the results of Section \ref{sectIBP}. \\

Similarly to $\exp(C)$ we define
\[ \exp(L):=\text{span}~\big\{  \sin((\eta,\cdot)_{\scriptstyle{L}}), \cos((\eta,\cdot)_{\scriptstyle{L}})~\big|~\eta \in L \big\}. \]
It holds $\exp(C^{\infty}) \subset \exp(L) \subset \mathcal{F}C_b^{\infty}(L)$, where $\exp(C)$ and $\mathcal{F}C_b^{\infty}(L)$ are defined in Section \ref{secintro}. \\
For $F \in \exp(L)$, and $f,h \in L$, define in analogy to (\ref{gateaux}) the directional derivative
\[ \partial_h F(f):=\lim_{\varepsilon \rightarrow 0} \frac{ F(f+\varepsilon h) - F(f)}{\varepsilon}.\]
The Fr{\'e}chet derivative $\nabla F:L \rightarrow L$ is given such that
\[ (\nabla F(f), h)_{\scriptstyle{L^2}}=\partial_h F(f).\]
Hence, for $\eta \in L$ holds
\[ \partial_h\sin((\eta,\cdot)_{\scriptstyle{L}})=(\eta,h)_{\scriptstyle{L}}~\cos((\eta,\cdot)_{\scriptstyle{L}}), \quad \nabla \sin((\eta,\cdot)_{\scriptstyle{L}})= \eta ~\cos((\eta,\cdot)_{\scriptstyle{L}}).\] 
Denote by $\mu^{BB}$ the image measure of the white noise measure $\mu$ under the map 
\[ BB: \mathcal{S}'(\mathbb{R}) \rightarrow L, \omega \mapsto BB(\omega):=(BB_t(\omega))_{0 \leq t \leq 1}, \]
i.e., $\mu^{BB}$ is the law of a Brownian bridge from $0$ to $0$ on $L=L^2(0,1)$ or equivalently, a centered Gaussian measure on $L$ with covariance operator $Q$ as defined in (\ref{covariance}). Note that $\mu^{BB}$ is well-defined, since $(BB_t)_{0 \leq t \leq 1}$ has surely continuous paths. Moreover, $BB$ is measurable with respect to $\mathcal{B}(L)$. This fact follows, since $BB$ is also well-defined as a $C[0,1]$-valued mapping, the embedding $C[0,1] \subset L^2(0,1)$ is continuous and $C[0,1]$-valued random variables are exactly given by real-valued processes on $[0,1]$ with continuous paths by \cite[Lemma 14.1]{Kal97}. \\
Then, the symmetric bilinear form $(\mathcal{G},\exp(L))$ is given by
\begin{align*} \mathcal{G}(F,G)&:= \frac{1}{2}~\int_{L} \big(\nabla F(f),\nabla G(f)\big)_{\scriptstyle{L}}~d\mu^{BB}(f) \\
&= \frac{1}{2}~\int_{\mathcal{S}'(\mathbb{R})} \big(\nabla F(BB(\omega)),\nabla G(BB(\omega))\big)_{\scriptstyle{L}}~d\mu(\omega) \quad \text{for } F,G \in \exp(L).
\end{align*}
$(\mathcal{G},\exp(L))$ is closable on $L^2(L;\mu^{BB})$ by \cite[Chapter 10]{DaPr06} and its closure is a quasi-regular local Dirichlet form, which we denote by $(\mathcal{G},W^{1,2}(L;\mu^{BB}))$. Moreover, the space $C_b^1(L)$ of continuous, bounded functions on $L$ with continuous, bounded Fr{\'e}chet derivative is contained in $W^{1,2}(L;\mu^{BB})$ and also
\[ \mathcal{G}(F,G)=\frac{1}{2}~\int_{L} \big(\nabla F(f),\nabla G(f)\big)_{\scriptstyle{L}}~d\mu^{BB}(f) \quad \text{for } F,G \in C_b^1(L)\]
in view of \cite[Proposition 10.6]{DaPr06}, where $\nabla F$ and $\nabla G$ denote the Fr{\'e}chet derivatives of $F$ and $G$ respectively. \\

Due to the behavior of the modulus at zero the zeros of the Brownian bridge are of particular interest regarding a further analysis. The following result states some properties of the zero level set of the Brownian bridge form zero to zero and can be found in \cite[Lemma 2.2]{Pet89} (see also \cite[Theorem 2.28]{MP10} for the corresponding result for Brownian motion):

\begin{lemma} \label{lemPet}
Denote by the random set
\[ z_{\omega}:= \{ t \in [0,1]~\big|~BB_t(\omega)=0\} \]
the zero-level set of the Brownian bridge. $z_{\omega}$ almost surely
\begin{enumerate}
\item has Lebesgue measure zero,
\item is closed,
\item has accumulation points at $t=0$ and $t=1$,
\item has no isolated points in $(0,1)$.
\end{enumerate}
\end{lemma}

Using these properties we can prove the following:

\begin{theorem} \label{thmW12}
$F \circ |\cdot| \in W^{1,2}(L;\mu^{BB})$ for every $F \in \mathcal{F}C_b^{\infty}(L)$ with weak derivative given by $\nabla (F \circ |\cdot|)=\textnormal{sgn}(\cdot)~\nabla F \circ |\cdot|$ and $\exp(C^{\infty})$ is dense in $W^{1,2}(L;\mu^{BB})$.
\end{theorem}

\begin{proof}
Assume that $F \in \mathcal{F}C_b^{\infty}(L)$, i.e., $F=g\big( (\eta_1,\cdot)_{\scriptstyle{L}},\dots, (\eta_k,\cdot)_{\scriptstyle{L}} \big)$ for $k \in \mathbb{N}$, $g \in C_b^{\infty}(\mathbb{R}^k)$ and $\eta_i \in L$ for $i=1,\dots,k$. Define $G:=F \circ |\cdot|$. By closedness and $C_b^1(L) \subset W^{1,2}(L;\mu^{BB})$ it is sufficient to construct a sequence $(G_n)_{n \in \mathbb{N}}$ in $C^1_b(L)$ with limit $G$ in $L^2(L;\mu^{BB})$ such that the sequence of derivatives converges in $L^2(L;\mu^{BB};L)$ to $\nabla G$ defined by 
\[ \nabla G(f):=\text{sgn}(f)~\nabla F(|f|)=\text{sgn}(f)~\sum_{i=1}^k \eta_i~\partial_i g((\eta_1,|f|)_{\scriptstyle{L}},\dots, (\eta_k,|f|)_{\scriptstyle{L}}) \quad \text{for } f \in L.\]
Let $B_n:\mathbb{R} \rightarrow [0,\infty)$, $n \in \mathbb{N}$, be given by
\begin{align*}
B_n(x) :=
   \begin{cases}
  -x - \frac{1}{2n}   & \text{for } x \leq - \frac{1}{n} \\
  \frac{n}{2} x^2   &  \text{for } -\frac{1}{n} < x < \frac{1}{n} \\
   x - \frac{1}{2n}  & \text{for } x \geq \frac{1}{n}.
   \end{cases}
\end{align*}
In particular, $B_n$ converges uniformly to $|\cdot|$ (note that $\Vert B_n - | \cdot |\Vert_{\infty}=\frac{1}{2n}$ and $|B_n(x)| \leq |x|$ for every $x \in \mathbb{R}$), $B_n$ is continuously differentiable and its first derivative converges pointwisely to $\text{sgn}$ for all points $x \neq 0$. Define
\[ G_n:=F \circ B_n \quad \text{for } n \in \mathbb{N}. \]
Furthermore, for $i=1,\dots,k$ let $(\eta^n_i)_{n \in \mathbb{N}}$ be the sequence in $L^{\infty}(0,1)$ given by
\[ \eta^n_i:= (\eta_i \wedge n) \vee (- n) := \max \{ \min \{ \eta_i, n \}, -n \} \] 
which converges to $\eta_i$ in $L$ by dominated convergence. In particular, it holds $|\eta_i^n | \leq n$ as well as $|\eta_i^n | \leq |\eta_i|$ $dx$-a.e. and $\eta_i^n \rightarrow \eta_i$ $dx$-a.e.. Set
\[ F_n:=g\big( (\eta^n_1,\cdot)_{\scriptstyle{L}},\dots, (\eta_k^n,\cdot)_{\scriptstyle{L}} \big) \]
as well as
\[ G_n(f):= F_n \circ B_n (f) = g\big( (\eta^n_1,B_n(f))_{\scriptstyle{L}},\dots, (\eta_k^n,B_n(f))_{\scriptstyle{L}} \big) \quad \text{ for } f \in L.\]
Note that $F_n \in  C_b^1(L)$ with Fr{\'e}chet derivative $\nabla F_n=\sum_{i=1}^k \eta_i^n~J_i^n$, where for $i=1,\dots, k$ $J_i^n:=\partial_i g((\eta^n_1,\cdot)_{\scriptstyle{L}},\dots, (\eta^n_k,\cdot)_{\scriptstyle{L}})$.
For $f,h \in L$ holds
\begin{align*}
\frac{B_n(f+\varepsilon h) - B_n(f)}{\varepsilon} \rightarrow h~ B_n^{\prime}(f) \quad \text{pointwisely $dx$-a.e. as } \varepsilon \rightarrow 0.
\end{align*}
The absolute value of the left hand side is bounded by $|h|$, since $|B'_n| \leq 1$. Thus, by dominated convergence holds for $i=1,\dots,k$
\[ \lim_{\varepsilon \rightarrow 0} \frac{1}{\varepsilon} (\eta_i^n, B_n(f+\varepsilon h)-B_n(f))_{\scriptstyle{L}}=(\eta_i^n, B_n^{\prime}(f) ~h)_{\scriptstyle{L}}=(B_n^{\prime}(f) ~\eta_i^n,h)_{\scriptstyle{L}} \]
which implies in use of the differentiability of $g$ that
\[ \lim_{\varepsilon \rightarrow 0} \frac{1}{\varepsilon} (G_n(f+\varepsilon h) - G_n(f))=\sum_{i=1}^k J_i^n(B_n(f))~(B_n^{\prime}(f)~\eta_i^n,h)_{\scriptstyle{L}}.\]
Hence, $G_n$ is G{\^a}teaux differentiable at every point $f \in L$ with derivative
\[ B_n^{\prime}(f) ~\sum_{i=1}^k \eta_i^n~J_i^n(B_n(f))= B_n^{\prime}(f) ~\nabla F_n(B_n(f)). \]
Moreover,  
\begin{align*}
&\Vert B_n^{\prime}(f)~ \nabla F_n(B_n(f)) - B_n^{\prime}(h)~ \nabla F_n(B_n(h)) \Vert_{\scriptstyle{L}} \\
\leq &\Vert \big( B_n^{\prime}(f) - B_n^{\prime}(h) \big)~ \nabla F_n(B_n(f)) \Vert_{\scriptstyle{L}} + \Vert B_n^{\prime}(h)~ \big( \nabla F_n(B_n (f)) - \nabla F_n(B_n(h)) \big) \Vert_{\scriptstyle{L}} \\
\leq & C(n) ~ \Vert B_n^{\prime}(f) -B_n^{\prime}(h) \Vert_{\scriptstyle{L}} + \Vert \nabla F_n(B_n (f)) - \nabla F_n(B_n(h)) \Vert_{\scriptstyle{L}} \\
\leq & C(n)~n~  \Vert f-h \Vert_{\scriptstyle{L}} + \Vert \nabla F_n(B_n (f)) - \nabla F_n(B_n(h)) \Vert_{\scriptstyle{L}}
\end{align*}
for a real constant $C(n)$ depending on $n$, since $\eta_i^n \in L^{\infty}(0,1)$ for $i=1,\dots,k$ and every $n \in \mathbb{N}$. Moreover, we used $|B_n^{\prime}(h)| \leq 1$ and the Lipschitz continuity of $B_n^{\prime}$ with Lipschitz constant $n$. Note that by the mean value theorem
\begin{align*}
|(\eta_i^n, B_n(f))_{\scriptstyle{L}}-(\eta^n_i,B_n(h))_{\scriptstyle{L}}| \leq \Vert \eta_i^n \Vert_{\scriptstyle{L}} ~ \Vert B_n(f)-B_n(h) \Vert_{\scriptstyle{L}} \leq \Vert \eta_i^n \Vert_{\scriptstyle{L}} ~ \Vert f-h \Vert_{\scriptstyle{L}}
\end{align*}
for $i=1,\dots,k$. This implies the continuity of $\nabla F_n \circ B_n$ and thus, the G{\^a}teaux derivative of $G_n$ is continuous at $f \in L$. Hence, $G_n \in C_b^1(L)$ with bounded Fr{\'e}chet derivative $\nabla G_n=B_n^{\prime}~ \nabla F_n \circ B_n$.\\
$(G_n)_{n \in \mathbb{N}}$ converges to $G=F \circ |\cdot |$ in $L^2(L;\mu^{BB})$. Indeed, for fixed $f \in L$ holds
\begin{align*}
\big|G_n(f)-G(f)\big| \leq \Vert \nabla g \Vert_{\infty} ~\sqrt{ \sum_{i=1}^k \big(\eta_i^n, B_n(f)-|f|\big)_{\scriptstyle{L}}^2} \rightarrow 0 \quad \text{ as } n \rightarrow \infty, 
\end{align*}
since $\eta_n (B_n(f)-|f|) \rightarrow 0$ pointwisely as $n \rightarrow \infty$ and $|B_n(f)-|f|| \leq 2|f| \in L$. Note that it even holds $|B_n(f)-|f|| \leq \frac{1}{2n}$. Furthermore, $|G_n(f)-G(f)| \leq 2 \Vert g \Vert_{\infty}$ for every $f \in L$. Therefore, we obtain again by dominated convergence
\[ \int_L \big|G_n(f) - G(f)\big|^2~ d\mu^{BB}(f) \rightarrow 0 \quad \text{ as } n \rightarrow \infty.\]
Similarly to $J_i^n$ define $J_i:=\partial_i g((\eta_1,\cdot)_{\scriptstyle{L}},\dots,(\eta_k,\cdot)_{\scriptstyle{L}})$ for $i=1,\dots,k$ such that $\nabla F=\sum_{i=1}^k \eta_i J_i$. It rests to show that $(\nabla G_n)_{n \in \mathbb{N}}$ converges to $\nabla G$ in $L^2(L;\mu^{BB};L)$ with $\nabla G(f)= \text{sgn}(f)~\nabla F(|f|)$ $ =\text{sgn}(f)~\sum_{i=1}^k \eta_i ~J_i(|f|)$ for $f \in L$. Indeed,
\[ \Vert \nabla G_n(f)-\nabla G(f) \Vert_{\scriptstyle{L}} \leq \sum_{i=1}^k \Vert B'_n(f)~ \eta_i^n ~J_i^n(B_n(f))- \text{sgn}(f)~ \eta_i ~J_i(|f|) \Vert_{\scriptstyle{L}} \]
and for $i=1,\dots,k$ it holds
\begin{align*}
&~ ~ ~\Vert B_n^{\prime} (f) ~\eta_i^n~J_i^n(B_n(f)) - \text{sgn}(f) ~\eta_i~J_i(|f|) \Vert_{\scriptstyle{L}} \\
&\leq \Vert ( B_n^{\prime}(f)-\text{sgn}(f))~\eta_i^n~J_i^n(B_n(f)) \Vert_{\scriptstyle{L}} + \Vert \text{sgn}(f) ( \eta_i^n~J_i^n(B_n(f)) - ~\eta_i~J_i(|f|) \Vert_{\scriptstyle{L}} \\
&\leq \Vert \partial_i g \Vert_{\infty} \Vert ( B_n^{\prime}(f)-\text{sgn}(f))~\eta_i^n \Vert_{\scriptstyle{L}} + \Vert \eta_i^n~J_i^n(B_n(f)) - ~\eta_i~J_i(|f|) \Vert_{\scriptstyle{L}} \\
&\leq  \Vert \partial_i g \Vert_{\infty} \Vert ( B_n^{\prime}(f) - \text{sgn}(f))~ \eta^n_i \Vert_{\scriptstyle{L}} +  \Vert \eta^n_i~(J_i^n(B_n(f)) - J_i(|f|)) \Vert_{\scriptstyle{L}} + \Vert (\eta^n_i -\eta_i)~J_i(|f|) \Vert_{\scriptstyle{L}} \\ 
&\leq  \Vert \partial_i g \Vert_{\infty} \Vert ( B_n^{\prime}(f) - \text{sgn}(f))~ \eta^n_i \Vert_{\scriptstyle{L}} +  \Vert \eta^n_i \Vert_{\scriptstyle{L}}~\Vert \nabla( \partial_i g) \Vert_{\infty}~\sqrt{\sum_{j=1}^k |(\eta^j_n,B_n(f)-|f|)_{\scriptstyle{L}}|} \\
&\hspace{9cm} + \Vert \partial_i g \Vert_{\infty}~\Vert \eta_i^n -\eta_i \Vert_{\scriptstyle{L}} \\ 
&\longrightarrow 0 \quad \text{as } n \rightarrow \infty
\end{align*}
for every $f \in L$ such that $dx ( \{ x \in [0,1]|~f(x)=0 \} ) =0$ by dominated convergence, since in this case we have that $(B_n^{\prime}(f)-\text{sgn}(f))~\eta_i^n \rightarrow 0$ $dx$-a.e. as $n \rightarrow \infty$ as well as $|(B_n^{\prime}(f)-\text{sgn}(f))~\eta_i^n| \leq 2 |\eta^n_i| \leq 2 |\eta_i| \in L$. Since it holds that $dx ( \{ x \in [0,1]|~f(x)=0 \} ) =0$ for $\mu^{BB}$-a.e. $f \in L$ due to Lemma \ref{lemPet}(i), it follows that
\[ \nabla G_n(f) \rightarrow \nabla G(f) \text{ in } L \text{ for } \mu^{BB} \text{-a.e. } f \in L. \]
Furthermore,
\begin{align*}
\Vert \nabla G_n(f)- \nabla G(f) \Vert_{\scriptstyle{L}}^2 \leq  2 \Vert \nabla G_n(f) \Vert_{\scriptstyle{L}}^2 + 2 \Vert \nabla G(f) \Vert_{\scriptstyle{L}}^2 &\leq  2 \Vert \nabla g \Vert_{\infty}^2 \sum_{i=1}^k \big(\Vert \eta^n_i \Vert_{\scriptstyle{L}}^2 + \Vert \eta_i \Vert_{\scriptstyle{L}}^2 \big) \\
&\leq 4 \Vert \nabla g \Vert_{\infty}^2 \sum_{i=1}^k \Vert \eta_i \Vert_{\scriptstyle{L}}^2 \in L^1(L;\mu^{BB}).
\end{align*}
This implies $\int_L \Vert \nabla G_n(f)- \nabla G(f) \Vert^2 ~d\mu^{BB}(f) \rightarrow 0$ as $n \rightarrow \infty$. Hence, $G= F \circ |\cdot| \in W^{1,2}(L;\mu^{BB})$ and $\nabla G= \text{sgn}(\cdot) ~\nabla F \circ |\cdot|$. \\
Furthermore, we know that $\exp(L)$ is dense in $W^{1,2}(L;\mu^{BB})$. Hence, in order to conclude density of $\exp(C^{\infty})$ it is sufficient to approximate $F=g((\eta,\cdot)_{\scriptstyle{L}})$, $\eta \in L$, $g \in \{ \sin,\cos\}$, by a sequence in $\exp(C^{\infty})$. This follows in use of the sequence given by $F_n:=g((\eta_n,\cdot)_{\scriptstyle{L^2}})$, $\eta_n \in C^{\infty}[0,1]$, $n \in \mathbb{N}$, such that $\eta_n \rightarrow \eta$ in $L$ by Lebesgue dominated convergence, boundedness of $g$, $g'$ and the fact that
\[ \int_L \Vert f \Vert_{\scriptstyle{L}}^2 ~d\mu^{BB}(f) < \infty. \]
\end{proof}

\section{Proof of the main result} \label{sectmain}

\begin{proof}[Proof of Theorem \ref{thmmain}]
First, let $F =g((\eta,\cdot)_{\scriptstyle{L}}) \in \exp(C^{\infty})$ for $\eta \in C^{\infty}[0,1]$, $g \in \{\sin,\cos\}$. Consider the exponential $\Phi_{z \eta}: = \exp(-\frac{1}{2} z^2 (Q\eta, \eta)_{\scriptstyle{L}}) \exp(z(\eta,\cdot)_{\scriptstyle{L}})$ as defined at the beginning of Section \ref{sectIBP} and $\tilde{\Phi}_{z \eta}:=\exp(z(\eta,\cdot)_{\scriptstyle{L}})=\exp(\frac{1}{2} z^2(Q\eta,\eta)_{\scriptstyle{L}})~ \Phi_{z\eta}$ with $z \in \mathbb{R}$. By Proposition \ref{propIBP} holds
\begin{align*}
& \mathbb{E}\big( \partial_{h \cdot \text{sgn}(BB)} \tilde{\Phi}_{z \eta}(BB)\big)\\
=&- \mathbb{E} \Big( \tilde{\Phi}_{z \eta} \int_0^1 h''_t~ |BB_t|~dt \Big) + \Big\langle \! \! \Big\langle \tilde{\Phi}_{z \eta}(BB), 2 \int_0^1 h_t ~\Gamma_t \diamond \delta_0 (|BB_t|)~dt \Big\rangle \! \! \Big\rangle,
\end{align*}
where the normalization factor is omitted. All terms are analytic in the variable $z$ and therefore, they extend in the natural way to an entire function on $\mathbb{C}$. In particular, for $z=i$ we can conclude that
\begin{align*}
& \mathbb{E}\big( \partial_{h \cdot \text{sgn}(BB)} F(BB)\big)\\
=&- \mathbb{E} \Big( F(BB) \int_0^1 h''_t~ |BB_t|~dt \Big) + \Big\langle \! \! \Big\langle F(BB), 2 \int_0^1 h_t ~\Gamma_t \diamond \delta_0 (|BB_t|)~dt \Big\rangle \! \! \Big\rangle
\end{align*}
by comparing real and imaginary part. By linearity the statement extends to arbitrary elements in $\exp(C^{\infty})$. In particular, $\int_0^1 h_t ~\Gamma_t \diamond \delta_0 (|BB_t|)~dt$ is well-defined on $\exp(C^{\infty})$ (see also Corollary \ref{coro}). Moreover, 
\begin{align} \label{eqhelp}
& \Big\langle \! \! \Big\langle F(BB), 2 \int_0^1 h_t ~\Gamma_t \diamond \delta_0 (|BB_t|)~dt \Big\rangle \! \! \Big\rangle
= \mathbb{E}\big( \partial_{h \cdot \text{sgn}(BB)} F(BB)\big) + \mathbb{E} \Big( F(BB) \int_0^1 h''_t~ |BB_t|~dt \Big)
\end{align}
for every $F \in \exp(C^{\infty})$. It holds
\begin{align*}
\left| \mathbb{E}\big( \partial_{h \cdot \text{sgn}(BB)} F(BB) \big)\right| 
\leq ~ \Vert h \Vert_{\scriptstyle{L}} ~ \mathbb{E} \big( \Vert \nabla F(BB) \Vert_{\scriptstyle{L}} \big),
\end{align*}
since $\partial_{h \cdot \text{sgn}(BB)}F(BB)=(h \cdot \text{sgn}(BB), \nabla F(BB))_{\scriptstyle{L}}$ $\mu^{BB}$-a.e. and $| \text{sgn}(BB) |=1$. Similarly,
\begin{align*}
\left| \mathbb{E} \big( F(BB) \int_0^1 h''_t~ |BB_t|~dt \big) \right|^2
\leq  \mathbb{E} \Big( \Big(\int_0^1 h''_t~ |BB_t|~dt \Big)^2 \Big) ~\mathbb{E} \big( F(BB)^2 \big) 
\leq  \frac{1}{6}~\Vert h'' \Vert_{\infty}^2 ~ ~\mathbb{E} \big( F(BB)^2 \big).
\end{align*}
Thus, by density of $\exp(C^{\infty})$ in $W^{1,2}(L;\mu^{BB})$ and continuity of the right hand side of (\ref{eqhelp}), the linear map $\int_0^1 h_t ~\Gamma_t \diamond \delta_0 (|BB_t|)~dt$ extends uniquely to an element in $(W^{1,2}(L;\mu^{BB})'$.\\
Let $h \in C_c^2(0,1)$ and $F \in \mathcal{F}C_b^{\infty}(L)$. By Theorem \ref{thmW12} it exists a sequence $(F_n)_{n \in \mathbb{N}}$ in $\exp(C^{\infty})$ which converges in $W^{1,2}(L;\mu^{BB})$ to $F \circ | \cdot |$. Hence, the limit
\begin{align} \label{limit} \Big\langle \! \! \Big\langle F(|BB|), 2 \int_0^1 h_t ~\Gamma_t \diamond \delta_0 (|BB_t|)~dt \Big\rangle \! \! \Big\rangle=\lim_{n \rightarrow \infty}\Big\langle \! \! \Big\langle F_n(BB), 2 \int_0^1 h_t ~\Gamma_t \diamond \delta_0 (|BB_t|)~dt \Big\rangle \! \! \Big\rangle
\end{align}
is well-defined and also independent of the sequence $(F_n)_{n \in \mathbb{N}}$. Moreover, (\ref{eqmain}) follows in use of
\begin{align*}
\mathbb{E}\big( \partial_{h \cdot \text{sgn(BB)}} F_n(BB)\big)&~= \mathbb{E}\big( \big(h \cdot \text{sgn}(BB),\nabla F_n(BB)\big)_{\scriptstyle{L}} \big) \\
&\rightarrow \mathbb{E} \big( \big(h \cdot \text{sgn}(BB),\nabla (F \circ |\cdot|)(BB)\big)_{\scriptstyle{L}} \big)=\mathbb{E}\big( \partial_h F(|BB|) \big)
\end{align*}
as $n \rightarrow \infty$.
\end{proof}

As a consequence of the first part of the proof of Theorem \ref{thmmain} we obtain the following corollary:

\begin{corollary} \label{coro}
Let $\eta \in C^{\infty}[0,1]$. Then, it holds 
\begin{align*}  &\Big\langle \! \! \Big\langle \sin((\eta,BB)_{\scriptstyle{L}}), 2 \int_0^1 h_t ~\Gamma_t \diamond \delta_0 (|BB_t|)~dt \Big\rangle \! \! \Big\rangle \\
&= \Im~ \Big\langle \! \! \Big\langle :\exp(\langle i K\eta, \cdot \rangle ):, 2 \int_0^1 h_t ~\Gamma_t \diamond \delta_0 (|BB_t|)~dt \Big\rangle \! \! \Big\rangle=0 
\end{align*}
as well as
\begin{align*}
&\Big\langle \! \! \Big\langle \cos((\eta,BB)_{\scriptstyle{L}}), 2 \int_0^1 h_t ~\Gamma_t \diamond \delta_0 (|BB_t|)~dt \Big\rangle \! \! \Big\rangle \\
&= \Re~ \Big\langle \! \! \Big\langle :\exp(\langle i K\eta, \cdot \rangle ):~, 2 \int_0^1 h_t ~\Gamma_t \diamond \delta_0 (|BB_t|)~dt \Big\rangle \! \! \Big\rangle \\
&= -2 \int_0^1 h_t \frac{\big((Q\eta)^{\prime}_t - \frac{1}{2} (Q\eta)_t \frac{1-2t}{t-t^2} \big)^2 + \frac{1}{4} \frac{1}{t-t^2}}{\sqrt{2 \pi (t-t^2)}}~\exp\big(\frac{((Q\eta)_t)^2}{2(t-t^2)}\big)~dt.
\end{align*}
\end{corollary}

\subsection*{Acknowledgment}
Financial support by the DFG through the project GR 1809/14-1 is gratefully acknowledged.

\bibliographystyle{alpha}
\bibliography{biblio}

\end{document}